\documentclass[11pt]{article}
\usepackage{color,indentfirst}
\usepackage{mathrsfs}
\usepackage{url}

\usepackage{graphicx}
\usepackage{graphics}
\usepackage{wrapfig}

\usepackage{amssymb,amsthm,amsmath,mathtools,enumitem,comment}
\usepackage{amsthm}
\usepackage{a4wide}
\theoremstyle{plain}
\newtheorem{thm}{Theorem}[section]
\newtheorem{lem}[thm]{Lemma}
\newtheorem{prp}[thm]{Proposition}

\theoremstyle{definition}

\newtheorem{dfn}{Definition}[section]

\let\oldmarginpar\marginpar
\renewcommand{\marginpar}[1]{\oldmarginpar{\scriptsize\texttt{\color{red}{#1}}}}

\numberwithin{equation}{section}

\newlength{\currentparskip} 
\newcommand{\ex}{\mathrm{e}}
\newcommand{\di}{\mathrm{d}}

\newcommand{\be}{\begin{equation}}
\newcommand{\ee}{\end{equation}}
\newcommand{\ben}{\begin{equation*}}
\newcommand{\een}{\end{equation*}}

\newcommand{\ba}{\begin{equation}\begin{aligned}}
\newcommand{\ea}{\end{aligned}\end{equation}}
\newcommand{\ban}{\begin{equation*}\begin{aligned}}
\newcommand{\ean}{\end{aligned}\end{equation*}}

 \date{}

\title{Geodesic random walks, diffusion processes  and Brownian motion on Finsler manifolds.}
\author{Tianyu Ma, Vladimir S.\ Matveev and Ilya Pavlyukevich}
\begin{document}
\maketitle

\begin{abstract}
We show that geodesic random walks on a complete Finsler manifold of bounded geometry converge to a diffusion process
which is, up to a drift, the Brownian motion corresponding to a Riemannian metric. 
\\[2ex]
\\[2ex]
\end{abstract}
\begin{small}
\tableofcontents
\end{small}
 
\section{Introduction.} 
Many processes in physics and natural sciences can be described with the help of random 
walks and their limit processes, the so-called diffusion processes. 
A possible philosophical explanation of this experimentally observed phenomenon is that the limit of random walks reflects the microscopic nature of the situation: Even 
fully deterministic microscopic systems can give rise to erratic seemingly 
random motions, practically indistinguishable from those produced by a stochastic process.

 Let us recall one of the first constructions of a random walk which is due to K.\ Pearson 
 in 1905 \cite{Pearson}. A more physically motivated approach is in the paper 
 \cite{Einstein} of A.\ Einstein from the same year.

We start from a point $p\in \mathbb{R}^2$, choose a random direction at the tangent space, go for distance $1$ along the straight line starting at this direction, and then repeat the procedure iteratively.  We obtain a stochastic process whose trajectories are 
piecewise-linear curves, see  Fig.~\ref{fig:0}. 

\begin{wrapfigure}{r}{0.4\textwidth}
\includegraphics[width=0.39\textwidth]{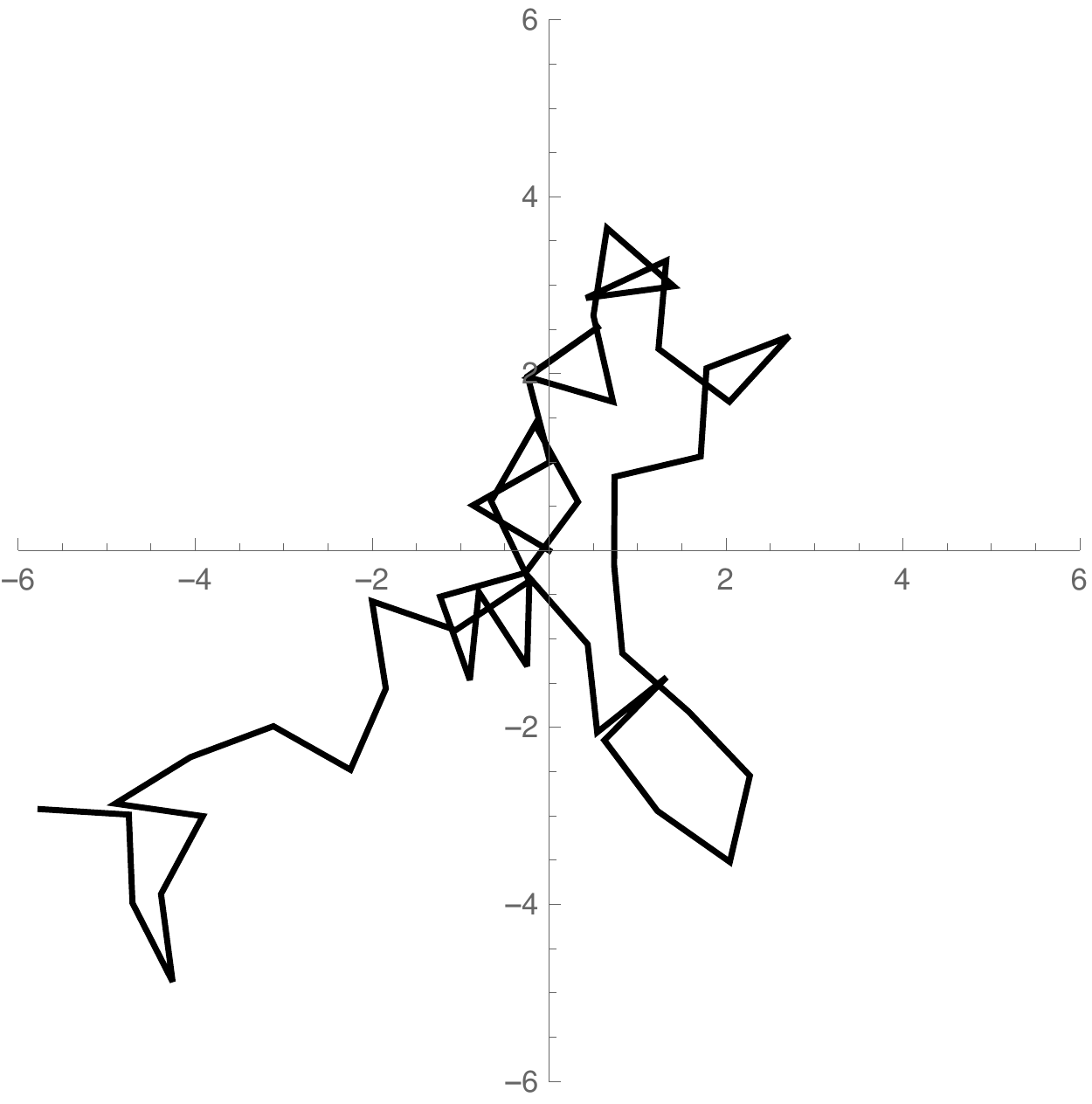}  
\caption{50 steps of a Pearson random walk. \label{fig:0}}
\end{wrapfigure}

 It is natural to  {\it renormalise} this process  
as follows: we assume that the steps have length not $1$ but $1/\sqrt{N}$ and we take $N$ steps in one unit of time.  
If the  procedure of choosing the random direction is invariant with respect to the isometry group of the flat $\mathbb{R}^2$ 
which was the case in \cite{Einstein,Pearson}, 
then by the Functional Central Limit Theorem the limit of this sequence as $N\to \infty$ exists and is the 
{\it (flat) Brownian motion}, see \cite[Chapter 2]{billingsley2013convergence}.

We see that in order to define  such a random walk, one needs two ingredients: the rule of choosing a random direction 
at a current position $p$ (i.e.,  a probability distribution $\nu_p$ 
on the space of tangent vectors at the point $p$) and an analogue of the notion of a  straight line, which describes the motion of a small particle with no external forces acting upon it.

In many systems in physics and natural sciences, small particles with no external forces acting upon them move along geodesics of a Finsler metric. We give necessary  definitions in \S \ref{sec:1.2.1}. Recall that geodesics are smooth curves and, similar to the straight lines,  
the initial point and the initial velocity vector determine the geodesic.  
The above definition of the random walk is immediately generalised to this case. 
Indeed, starting from a point $p$ of a Finsler manifold $(M,F)$  such that every tangent space $T_pM$ is equipped with a probability measure  $\nu_p$, choose a random vector $v$ in the tangent space, go the distance $F(v)/\sqrt{N}$ along the geodesic starting at $p$ with the  initial 
velocity $v$ and then repeat the procedure (if $\nu_p$ is not centered we rescale it as in \S \ref{s:RRW}). We obtain a stochastic process whose trajectories are piecewise geodesic curves (e.g., they are  glued together from geodesic segments). 

The present paper studies such geodesic random walks on Finsler manifolds and their limit diffusion processes, and concentrates on the 
fundamental question of the existence and uniqueness of the limit process.
Our main result is that under assumptions natural from the viewpoint of Finsler geometry,
the limit process exists and is unique. Moreover, it is a diffusion process whose 
generator is a non-degenerate elliptic second order partial differential operator for which we give a precise formula.  

The Riemannian version of our result (recall that Riemannian metrics are Finsler metrics) was obtained e.g.\ by E.\ J\o{}rgensen \cite{Jorgensen}.   

Geodesic random walks on Finsler manifolds and their limit processes are 
 of course  natural topics from the viewpoint of both differential geometry and theory of stochastic processes. They may have applied interest since Finsler manifolds are used to model different physical situations with anisotropies at an infinitesimal level, see e.g.\ \cite{Antonelli1,Antonelli2,Caponio, Cvetic,Gibbons, Hohmann, Markvorsen, Pfeifer,  Yajima}, and may also be used for certain models in information geometry, see e.g.\ \cite{shen}.

Although Brownian motions and diffusion processes on Finsler 
manifolds were discussed in the literature (see  e.g.\ the 
books \cite{Antonelli3,Wilke}), the very basic question of the existence and uniqueness of the limit process for Finsler geodesic random walks has not been  rigorously treated.

More precisely, the work \cite{Wilke} on Finsler Brownian motions goes in the other direction: 
It starts from a stochastic differential equation which is constructed by a Finsler metric $F$, a volume form $\mu$, and an extra data $u_0\in H^1_0(M)$ on $M$. 
It is easy to see that for a generic Finsler metric, solutions of this stochastic differential equation do not  correspond to a limit process of a sequence of geodesic  random walks.  

This general approach, in which one starts with an elliptic differential operator (or a Dirichlet form)
in order to construct a diffusion process, is a very popular and powerful approach to diffusion processes on metric spaces. It allows in particular   
  to treat the case of non-smooth background metric structures, see e.g.\ \cite{Gigli,Kuwae,sturm1998diffusion}. 
  This approach does not ensure that the resulting stochastic process is the  limit process of a sequence of random walks. If the background is almost Riemannian 
  (say, Alexandrov with bounded curvature, as in \cite{Gigli} and \cite{Kuwae}), the best one can do is to relate random walks on the Riemannian spaces approximating our metric space to the diffusion process on our metric space. These results  cannot be applied in the  Finslerian situation, since Finsler metrics cannot be approximated by Riemannian metrics. Our results will possibly allow to extend this group of methods to a 
  Finslerian situation and we plan to do this in our future works.

Let us now discuss the corresponding results of the book \cite{Antonelli3}, where  many different approaches of 
constructing different non-equivalent diffusion processes  (on the manifold or on the tangent bundle to the manifold) 
by a Finsler metric are suggested.  One of these approaches (see \cite[\S A2]{Antonelli3}) is seemingly close to  
ours, and considers the limit processes of Finsler geodesics random walks 
(in their case, the distribution $\nu_p$ is quite special and is canonically constructed by the Finsler metric). 
Unfortunately no rigorous proof of convergence is given: it is merely  claimed  that the limit process exists and is unique, and referred  to \cite{Pinsky1,Pinsky2} for methods and  technical details. 

The references  \cite{Pinsky1,Pinsky2} are mostly survey papers about  geodesic random walks on  Riemannian manifolds. 
The  methods discussed there assume  and rely on the special form of the probability measure $\nu_p$ on tangent spaces. Moreover, it is assumed that the Riemannian manifold  is {\it stochastically complete}. The property of stochastic completeness is a nontrivial property, and examples show that not all complete manifolds are stochastically complete. In the Riemannian case, there is 
a number of criteria of stochastic completeness, see  e.g.\ 
\cite{Grigoryan, Yau}. In particular, if the Ricci curvature of a complete Riemannian manifold is bounded from below, the manifold  is stochastically complete. In the Finslerian situation, we did not find any relevant works on  stochastic completeness and the claim of  \cite[\S A2]{Antonelli3}  that the  methods of   \cite{Pinsky1, Pinsky2} can easily be applied in the Finslerian situation looks overoptimistic.

Note that as a by-product, we have proved that every  complete Finsler 		manifold of {\it bounded geometry}  (see Definition  \ref{definition:bounded geometry})   is stochastically complete; that is,  the  limit process of   Finsler geodesic random walks  is stochastically complete in the sense of \cite[\S 4.2]{Hsu}. It is interesting to try to relax the assumption of bounded geometry in this statement and we plan to do this in future works. 		 

A very successful   approach to geodesic random walks and diffusion processes on  Riemannian  manifolds, 
which allows essential freedom in the choice of the probability measures $\nu_p$, is in  \cite{Jorgensen}. 
Many arguments in \cite{Jorgensen}   are based on the following property
which holds in the Riemannian but not in the Finslerian case: Consider an  arc-length parametrized 
geodesic segment $\gamma\colon [0,\varepsilon]\rightarrow M$ of a (smooth) Riemannian metric. Take a vector $v\in T_{\gamma(0)}M$ of length one and its parallel transport $v_\varepsilon\in  T_{\gamma(0)}M$ along the geodesic segment. Next, consider the  arc-length parametrised geodesic geodesics $\gamma_{v}$ and $\gamma_{v_\varepsilon}$ 
which start from $\gamma(0)$ and $\gamma(\varepsilon)$ with the initial vectors $v$ and $v_\varepsilon$, respectively. 
Then the distance between $\gamma_{v}(t)$ and $\gamma_{v_\varepsilon}(t)$, 
behaves, 
for $\varepsilon \to 0$ 
and $t\to 0$, as $\varepsilon(1+C t^2)$.
In the Euclidean case, the distance does not depend on $t$ at all and is equal to $\varepsilon$.  
In the Finslerian situation, this property does not hold for a generic metric and 
a straightforward generalisation of \cite{Jorgensen} is not possible. 

In this paper we prove that under the assumptions natural from the viewpoint 
of Finsler geometry (everything is smooth, 
the manifold is complete and has bounded geometry), the sequence of geodesic random walks
converges to a unique diffusion process, see Theorem \ref{thm:convergence main theorem}. 
Moreover, we show that the generator of this diffusion process is an elliptic operator, and give an integral formula for its coefficients.

As explained above, the generator of the limit diffusion process is a non-degenerate elliptic operator.
If the probability measure $\nu_p$ on each $T_pM$ is constructed by $F_{|T_pM}$ (we give examples in \S \ref{sec:average}) then
this elliptic operator is a natural candidate for a Beltrami--Laplace operator of the Finsler metric. 
Note that, different from the Riemannian case, there exist many different 
Finslerian analogues of the Beltrami--Laplace operator. 
We refer to \cite{antonelli2012theory}, where many different constructions of the Riemannian 
Beltrami--Laplace operator are  mimicked in  the Finslerian setting. In the Riemannian case they 
all give the same Beltrami--Laplace operator. In the Finslerian case one obtains different  operators.  
Most operators in \cite{antonelli2012theory} are linear but  there also exist nonlinear versions of the Finslerian Betrami--Laplace operators, see e.g.\  \cite{Ohta,shen_laplacian}. 

An interesting by-product of our result is that the generator of the limit diffusion 
process corresponding to Finsler geodesic 
random walks coincides, up to first order terms (the so-called ``drift''), with that of a {\it Riemannian Brownian motion}. 
 This result of us explains why it is hard or even impossible to experimentally distinguish a diffusion 
process coming from a Riemannian metric from that of coming from a Finsler metric. See \S \ref{sec:1.4} for more details.

Naturally, the topic of this paper, and therefore also the methods of the proof, 
belong both to differential geometry and to the theory of stochastic processes. The group of the methods coming from stochastic processes 
is actually standard for this type of problems (though nontrivial) and was understood at least in the 70th-80th.   
The novelty which allowed to solve this natural and actively attacked  problem came from Finsler geometry, 
and the key lemma  is Lemma \ref{lemma:generator bounds}, whose proof uses a nontrivial and not widely known result of \cite[\S 15]{shen2001lectures}.

\subsection*{Acknowledgements.} We thank M.\ von  Renesse for useful discussions.
T.M.\ and  V.M.\ thank the DFG for the financial support (Einzelprojekt MA 2565/6).

\section{Setting and results. } \label{sec:1.2}

\subsection{Finsler manifolds.}\label{sec:1.2.1} 

First we recall the basic definitions in Finsler geometry. Let $M\coloneqq M^m$ be a  
$m$-dimensional manifold, $m\geq 1$. 
Suppose that $(x^1,\dots,x^m)$ is a local coordinate at some $p\in M$. Then $y_i=\partial x_i$ induces a local coordinate $(x^1,\dots,x^m,y^1,\dots,y^m)$ on $TM$. 
For simplicity, for a function $H\colon TM\to \mathbb R$ we use the notations  $H_{x^i}=\partial_{x^i}H$ and $H_{y^i}=\partial_{y^i}H$.

A smooth \emph{Finsler manifold} $(M,F)$ is a smooth manifold $M$ together with a 
continuous function $F\colon TM\to \mathbb{R}_{\geq 0}$ called the \emph{Finsler metric} (Finsler function) satisfying the following conditions:
\begin{description}
\item[Regularity:] The function $F$ is smooth on $TM\setminus \lbrace 0\rbrace$.
\item[Positive Homogeneity:] For any $(x,y)\in T_xM$ and $\lambda\geq 0$, we have $F(x,\lambda y)=\lambda F(x,y)$.
\item[Strong Convexity:] For $0\neq (x,y)\in T_xM$, the \emph{fundamental tensor} defined by
\begin{align}
\label{e:FT}
[g_{(x,y)}]_{ij}\coloneqq\left( \dfrac{1}{2} F^2 \right)_{y^iy^j}
\end{align}
is strictly positive definite.
\end{description}

 The indicatrix bundle of $(M,F)$ is defined by
\begin{align*}
IM=\lbrace Y\in TM\colon F(Y)=1\rbrace.
\end{align*}
For any $p\in M$, the fibre $I_pM$ of $IM$ is a convex hypersurface  in  $T_pM$ diffeomorphic  to $\mathbb{S}^{m-1}$.\\

If $(M,\mathbf{g})$ is a Riemannian manifold, one can naturally endow it with a Finsler metric by setting 
$F(Y):=\sqrt{\mathbf g(Y,Y)}$, $Y\in TM$. Conversely, a Finsler function corresponds to some Riemannian metric $\mathbf g$ if and only if its fundamental 
tensor $g_{ij}$ defined in \eqref{e:FT} depends only on the $x^i$-variables. 

 The definitions of geodesics and exponential maps can be  naturally generalised to the Finslerian situation.   
A smooth curve $\gamma\colon [a,b]\rightarrow M$ is a {\it geodesic},  if it is a stationary point of the energy functional 

\begin{equation} \label{eq:F2}
E[\gamma]\coloneqq \dfrac{1}{2}\int_a^b F^2(\gamma(t),\dot{\gamma}(t))\, \di t.
\end{equation}
 among all piecewise smooth curves starting at $\gamma(a)$ and ending at $\gamma(b)$.
It is known that for any $p\in M$ and for  any $Y\in T_pM$, there exists  a unique geodesic $\gamma_Y=\gamma_Y(t)$ such that $\gamma(0)=p$ and $\dot{\gamma}(0)=Y$. We define the {\it exponential map}  at $p$ to be
\begin{align}
\exp_p\colon T_pM\ni Y\mapsto \gamma_Y(1)\in M
\end{align}
for all $Y\in T_pM$ such that $\gamma_Y(t)$ is defined for $t\in[0,1]$.
We say $(M,F)$ is {\it forward complete} if for any $p\in M$ the exponential map $\exp_p$ is defined for all $Y\in T_pM$. The manifold 
$(M,F)$
is {\it geodesically complete}, if each geodesic $\gamma$ can be extended to a geodesic defined for all $t\in (-\infty,\infty)$.

For a piecewise smooth curve $\gamma\colon [a,b]\rightarrow M$, its {\it length}  is defined by
\begin{align}
\mathbf{Length}(\gamma)=\int_a^b F(\gamma(t), \dot\gamma(t))\,\di t.
\end{align}
The Finsler function $F$ defines the following {\it asymmetric} and  {\it symmetrized} distances on $M$:
\begin{align}
\notag
d_a(p,q) &\coloneqq  \inf\Big\lbrace \textbf{Length}(\gamma)\colon \gamma \text{ is a piecewise smooth curve from $p$ to $q$}\Big\rbrace, \\
\label{eqn:symmetrized distance}
d(p,q)&\coloneqq \max\lbrace d_a(p,q),d_a(q,p)\rbrace
\end{align}
 By the Hopf--Rinow theorem for Finsler manifolds (see e.g.\ \cite[Section 6.6]{bao2000introduction}), 
if $(M,F)$ is forward complete, the metric space $(M,d)$ is complete. For a forward complete $(M,F)$, every closed ball of $(M,d)$ is compact.
The manifold $M$ can be naturally endowed with the Borel sigma-algebra that makes it a measure space. 

Like in  the  Riemannian case, geodesics of  Finsler metrics  are local distance minimizing (with respect to $d_a$) curves. The formula \eqref{eq:F2}    ensures that they are parametrised proportional to the arc-length parameter.  
Note, as $F$ is in general not \emph{reversible}, i.e.\ $F(x,y)\not\equiv F(x,-y)$, 
the distance function $d_a$ and geodesics are not reversible 
either. 

We will assume below that the flag and $T$-curvatures (the definitions are in e.g. \cite{shen2001lectures}) of our Finsler manifold are uniformly  bounded.   The flag curvature $K$ can be thought as a generalisation of the Riemannian sectional curvature. 
The definition of  $T$-curvature  (see \cite[\S 10.1]{shen2001lectures})  is essentially Finslerian since it vanishes for Riemannian manifolds.

Within the whole paper we assume  the following set of hypotheses.

\medskip
\noindent
\textbf{H}$_{c}$: The manifold $(M,F)$ is connected and  forward complete.

\medskip
\noindent
\textbf{H}$_{b}$: The manifold $(M,F)$ has bounded geometry in the following sense: 

\begin{dfn}
\label{definition:bounded geometry}
We say a Finsler manifold $(M,F)$ has \emph{bounded geometry} if the followings hold:
\begin{enumerate}
\item   \label{condition:uniform elliptic} Uniform ellipticity: There is some constant $C>1$ such that for any $p\in M$ and any non-zero $u,v\in T_pM$, we have 
\begin{align}
\frac{1}{C^2}F^2(v)=\dfrac{1}{C^2}g_v(v,v)\leq  g_u(v,v)\leq C^2g_v(v,v)=C^2F^2(v).
\end{align} 
\item  \label{condition:K-bound} The flag curvature $K$ is bounded uniformly and absolutely by some constant $\lambda>0$, namely $\lVert K\rVert \leq \lambda$.
\item  \label{condition:T-bound} The $T$-curvature is also bounded uniformly and absolutely in the following sense. For any $p\in M$, any $u,v\in T_pM$ with $F(v)=1$, the $T$-curvature satisfies
\begin{align}
\label{e:T}
\lvert T_v(u)\rvert\leq \lambda\lbrace g_v(u,u)-[g_v(u,v)]^2\rbrace
\end{align} 
\end{enumerate}
\end{dfn}

Note that all objects   used in the definition of ``bounded geometry'' are microlocal,  and for an explicitly given Finsler metric it is possible   to check whether it has bounded geometry. Moreover, if $M$ is compact, then every smooth Finsler metric on it has bounded geometry. 

In this paper we use the following notations.

We say a function $f\colon M\rightarrow \mathbb{R}$ vanishes at infinity, if $\forall \varepsilon>0$, there exists some compact set $K_{\varepsilon}\subset M$ such that $\Vert f\Vert\leq \varepsilon$ outside $K_{\varepsilon}$.

We denote the unit discs on $TM$ by
\begin{align*}
D_pM=\lbrace Y\in T_pM \colon F(Y)\leq 1\rbrace.
\end{align*}

Let $d$ be the symmetrized distance defined by \eqref{eqn:symmetrized distance}. We  denote  the open balls by: 
\begin{align*}
B_p(\varepsilon)=\lbrace q\in M\colon d(p,q)<\varepsilon\rbrace ,\quad \varepsilon>0.
\end{align*}

In this paper $\mathcal{B}$ is the space of Borel measurable real valued functions on $M$, $\mathcal{C}_0$ is the space of 
continuous that vanish at infinity,
$\mathcal{C}^\infty$ is the space of smooth functions,
$\mathcal{C}^\infty_K$ is the space of  smooth functions with compact support. Furthermore,
$D([0,\infty),M)$ is the collection of right continuous functions $\gamma\colon [0,\infty)\rightarrow M$ with left limits, and 
$C([0,\infty),M)$ is space of continuous functions $\gamma\colon [0,\infty)\rightarrow M$.

\subsection{Rescaled geodesic random walks.\label{s:RRW}}  
 
As a motivation for  ``rescaling'', 
let us consider the following  example of a geodesic  random walk: The  manifold is  $\mathbb{R}$  with 
the standard flat metric and  $\nu_p$ is defined as follows. At every point $p$, the support of $\nu_p$ 
is two vectors 
$\{1, -1\}\subset T_p\mathbb{R}=\mathbb{R}$ such that the probability of $-1$ is $1/4$, and 
probability of $1$ is $3/4$. That is,  the particle goes 
  the distance $1/\sqrt{N}$
with probability $1/4$  in the negative direction  and with probability $3/4$  in the positive direction.  

 In this case the mean value of the position of the particle in one step is $1/(2 \sqrt{N})$ in the positive direction, so in $N$ steps the mean value is $\sqrt{N}/2$. For $N \to \infty$, most trajectories  ``escape to infinity'', 
 so the limit of the sequence
 of such 
stochastic  processes for $N\to \infty$ does not exist. 

This phenomenon appears in all dimensions 
when the probability distribution $\nu_p$ (on $T_pM$) has a nonzero mean 
(note that almost every  trajectory of the standard flat Brownian motion is not a rectifiable curve). 

Because of this, we introduce below  the {\it rescaled } random walk (we will formalize this definition in \S \ref{sec:form1}). We denote  the mean of $\nu_p$  by $\mu_p$, 
and  modify  the measure $\nu_p$ by shifting  it in $T_pM$   by  
the vector  $-\mu_p + \mu_p/\sqrt{N}$,\ i.e.\ $\tilde \nu_p(y):=\nu_p(y-\mu_p + \mu_p/\sqrt{N})$.  
We will call this operation the {\it rescaling of measure}.

Let us explain this operation.  First we note that the easiest would be to shift the measure by the vector $-\mu_p$.
This will make the mean of the new measure equal to $0$, and the phenomenon demonstrated in the example above does not appear. 
Note that some papers   on random  walks on  Finsler metric, for example \cite{Antonelli3},
assume that   both Finsler metric and the measure $\mu_p$ are centrally-symmetric on every $T_pM$ 
(the so-called reversible situation). We do not want to 
do it since most examples of Finsler metrics appearing in applications are not reversible. 

This  rescaling of the measure  was used in the Riemannian situation by E.\ J\o{}rgensen in \cite{Jorgensen}, and his motivation,
which is also valid in our situation, was  that for $N=1$ the increments of the 
random walk should be distributed according to $\nu_p$.  Indeed, for $N=1$ we have $-\mu_p + \mu_p/\sqrt{N}=0$. 
We also feel that in the case that Finsler geodesic random walk is used to describe a physical model,  
the mean of $\nu_p$ should  somehow come in the definition. Of course, it may come with any other coefficient $\alpha$, 
i.e., $\nu_p$ may be  shifted  by the vector $-\mu_p + \alpha \mu_p/\sqrt{N}$. 
But also in this case (even if $\alpha$ depends  on the position) our results are applicable. 
Indeed, if we modify $\nu_p$ by shifting it by $\beta\mu_p$, then the rescaled measure will be shifted by 
$(1+\beta)\mu_p/\sqrt{N}$.  Thus, all results of our paper  can be applied for any $\alpha$, in particular for $\alpha=0$.

We also expect that such rescaling of the measure is physically-relevant, since microscopic particles can not make too long jumps because of friction and collisions; so even if the probability of the particle to go to the ``right'' is higher than the probability to go 
to the ``left'',  the particle  does not escape  to infinity in short time, contrary to what is suggested by the  random walk described in the beginning of this section.   

\subsection{The main result.} \label{sec:main} 

We will assume that the Finsler manifold $(M, F)$ is connected and forward  complete (Hypothesis \textbf{H}$_{c}$) and has bounded geometry (Hypothesis  \textbf{H}$_b$). In addition, we make the following assumption on the family of measures $\{\nu_p\}$.

\medskip

\noindent
\textbf{H}$_\nu$: We assume that $\nu= \{\nu_p\}$ is a smooth family of probability measures inside   $DM:=\{ Y \in TM  \mid F(Y)\le  1\}$ in $TM$ or on the $F$-indicatrices.

In the first case we require that $\nu$ is a smooth $m$-form on  $DM$ such that  for every $p$,  
   the restriction  $\nu_p:=\nu|_{T_pM}$ 
   is a form on the disc $D_pM:= \{ Y \in T_pM  \mid F(Y)\le  1\}$  inducing a probability measure. 
Similarly,  in the second case
$\nu$ is a smooth $(m-1)$-form on $IM$ such that for each $p\in M$ the restriction $\nu_p:=\nu|_{I_pM}$ is a probability measure.

\medskip

This hypothesis is very natural  from the viewpoint of  Finsler geometry, and covers many   choices that have their natural counterparts in the Riemannian setting; we will give a few examples in \S \ref{sec:average}.

Our main result is the following theorem: 
 
\begin{thm}
\label{thm:convergence main theorem}
Let Hypotheses $\mathbf{H}_{c}$, $\mathbf{H}_b$ and $\mathbf{H}_\nu$ be satisfied.
%
Consider  a family of geodesic random walks starting at $p_0$ constructed from $(M,F)$ and $\lbrace \nu_p\rbrace_{p\in M}$. Then, this  sequence  
has a unique weak limit $\xi$. The process $\xi$ is a diffusion whose generator is a 
non-degenerate elliptic differential operator $A$ with smooth coefficients given by 
\begin{align}
\label{eqn:diffusion generator}
Af(p)=\di f(\mu_p)+\dfrac{1}{2}\int_{T_pM} \dfrac{\di^2}{\di t^2}\biggr\vert_{t=0}f\circ \gamma_{Y-\mu_p}(t)\nu_p (\di Y),\quad f\in \mathcal{C}^{\infty}_K.
\end{align}
Here $\gamma_{Y-\mu_p}$ is the geodesic with initial vector $Y-\mu_p$. In the local coordinates, it has the following form:
\begin{equation}
\begin{aligned}
Af(p)=& f_k\left(\mu_p^k-\dfrac{1}{2}\int_{T_pM}\Gamma^k_{ij}\left(p,y-\mu_p \right)(y^i-\mu_p^i)(y^j-\mu_p^j)\, \nu_p(\di y)\right)\\
& +\dfrac{1}{2} f_{ij}\int_{T_pM}(y^i-\mu_p^i)(y^j-\mu_p^j)\, \nu_p (\di y),
\end{aligned}
\end{equation}
where $\Gamma^k_{ij}$ are the formal Christoffel symbols of the second kind given by
\begin{align*}
\Gamma^k_{ij}(x,y)=\dfrac{1}{2}g^{ks}\left(\dfrac{\partial g_{is}}{\partial x_j}+      \dfrac{\partial g_{js}}{\partial x_i}-\dfrac{\partial g_{ij}}{\partial x_s} \right)(x,y),\quad y\neq 0,
\end{align*}
 $f_k= \partial_{x^k} f$  and $f_{ij} = \partial^2 _{x^ix^j} f$.

Moreover, $\xi$ is stochastically complete.
\end{thm}

\subsubsection{Limit diffusion as a Riemannian Brownian motion with drift.\label{sec:1.4}}   

Recall that  {\it the Riemannian Brownian 
motion} is   a  diffusion process   which is 
 the limit of geodesic random walks with identically distributed steps.  Here \emph{identically} distributed should be understood as follows:   the probability measure $\nu_p$ is invariant with respect to  the parallel transport along any curve and is invariant with respect to the standard action of $SO(\mathbf{g})$ on $T_pM$. Actually, for a generic metric  invariance with respect to the parallel transport implies $SO(\mathbf{g})$-invariance.

It is known that   
the  generator  of a  Riemannian Brownian   is proportional to the Beltrami--Laplace operator of the metric, so its symbol 
is proportional with a constant coefficient to  the inverse of the Riemannian metric. 

By Theorem \ref{thm:convergence main theorem},  in the Finslerian  case  the generator $A$ of  the limit process of the geodesic random walk is a second order non-degenerate elliptic differential operator on $M$. 
Hence the symbol $\sigma(A)$ of $A$ is dual to a Riemannian 
metric on $M$ which we denote $\mathbf{g}_A$. Then the Beltrami--Laplace operator $\Delta^A$ of $\mathbf{g}_A$ and 
$A$ have the same symbol. Hence $A-\Delta^A$ is just a vector field on $M$. We call this vector field the {\it drift} of $A$. In the Riemannian case the drift is always zero.

In particular,  though Finsler metrics are much more complicated than Riemannian metrics, one almost 
does not see the difference on the level of 
 diffusion processes (only first order terms of generators may be different).  
This  should  be the  reason why  Finslerian  effects related to 
diffusion   were not observed experimentally in   physical or natural science systems, even  in those where  the free motion of particles   
corresponds  to  geodesics of a certain Finsler metric.   See e.g.\ \cite{Florack} where in a  highly anisotropic situation (diffusion weighted magnetic resonance imaging of brain), the measurement returned a Finsler metric which is very close to a Riemannian metric. 

This observation may provide additional mathematical tools for natural science and physics. Indeed, in most cases   the probability distributions $\nu_p$ can be ``read'' from the description of the model 
(in fact in many cases they are generated by the volume form of the standard flat metric). 
Empirical observations of diffusions may provide tools for testing mathematical models of the system in 
question or determining their parameters.

\subsubsection{Canonical constructions of Riemannian metrics.  } 

\label{sec:average}

In the Riemannian situation, there is essentially only one canonical (i.e., coordinate-invariant) 
construction of a probability measure on $T_pM$. Indeed, coordinate-invariance of the construction implies 
that the metric is invariant under the group $SO(\mathbf{g})$, which implies that in the orthogonal coordinates 
$(y^1,\dots,y^m)$  on $T_pM$ it is given by $\phi((y^1)^2+\cdots +(y^m)^2 )\, \di y^1\wedge \cdots \wedge \di y^m$. 
The function $\phi$ is the same for all points $p$, has the  property that it is nonnegative and that the integral 
$\int_{\mathbb{R}^n} \phi((y^1)^2+\cdots +(y^m)^2 )\, \di y^1\wedge  \cdots  \wedge \di y^m =1$.

In the Finslerian situations there are many natural non-equivalent constructions of a measure on $T_pM$. Let us recall the following three.

\begin{description} 
\item[Measure coming from the Lebesgue measure:]
For any $p\in M$, let $\omega'_p=\di y^1\wedge\cdots\wedge\di y^m$ be a Lebesgue measure on $T_pM$. It is known that it is unique up to a positive coefficient. 
We restrict it to the ball $D_pM$ (that is, the measure of an open set  $U\subset T_pM$ it the Lebesgue-measure  of the intersection $D_pM \cap U$), and normalize it  such that it becomes a probability measure.

\item[Measure coming from the fundamental tensor:]
For any $p\in M$, the fundamental tensor $g_{ij}$ defines a Riemannian metric on the compact manifold $I_pM$. 
Normalizing the volume on $I_pM$ induced by $g_{ij}$ we obtain probability measure 
\begin{align}
\nu_p\coloneqq\dfrac{\mathrm{vol}_g}{\mathrm{vol}_g(I_pM)}.
\end{align}
This probability measure is close to the one  used  in \cite[\S A2]{Antonelli3}.

\item[Measure coming from the Hilbert form:]
Denote $\mathbb{P}^+(M)$ the positive projectivized tangent bundle. The Hilbert 1-form $\hat{\omega}=F_{y^i}\,\di x^i$ defined on 
$TM\setminus \lbrace 0\rbrace$ is actually a pull back of some 1-form $\omega$ on $\mathbb{P}^+(M)$ by the standard projection.
It is well known that $\omega\wedge(\di \omega)^{m-1}$ defines a volume form on $\mathbb{P}^+(M)\simeq IM$. Let 
$\mathrm{i}_p\colon I_pM\rightarrow IM$ be the standard inclusion
and $\pi\colon IM\rightarrow M$ be the canonical projection. It is known (see e.g.  \cite{barthelme2013natural}) that  there exists a $(m-1)$-form $\alpha^F$ on $IM$ and a volume form $\omega^F$ on $M$ such that 
$\alpha^F\vert_{I_pM}$ is a unique volume form on $I_pM$ for each $p\in M$ with
\ba
&\mathrm{vol}_{\alpha^F}(I_pM)=1,\\
&\alpha^F\wedge \pi^{*}\omega^F=\omega \wedge(\di \omega)^{m-1}.
\ea
Hence we can take $\nu_p\coloneqq \mathrm{vol}_{\alpha^F}$ on $I_pM$.
\end{description}

Each of these measures satisfies the Hypothesis $\mathbf{H}_\nu$,  and is coordinate-independently constructed from  $F$. 
In the case the Finsler metric is reversible, the dual of the symbol of the generator 
corresponding to the first measure gives the {\it  Binet--Legendre}
metric (see e.g.  \cite{Centore, MT2012}). The second choice of the measure  gives  the {\it averaged metric }  used in \cite{berwald,MRTZ} (a small modification of the construction  leads to the metric from \cite{Szaba}), and the third choice of measure generates the {\it Finsler Laplacian}   from  \cite{barthelme2013natural}. Note that 
  the Binet--Legendre metric, the averaged metric, and the Finsler Laplacian from  \cite{barthelme2013natural} appear to be   effective tools for solving different problems in Finsler geometry; we expect that other natural choices of the measure 
  $\nu_p$ may  also be useful in Finsler geometry.

\subsection{Example: Limit diffusion for a Katok Finsler metric.} \label{sec:Katok}
Let $(\mathbb{S}^2,\mathbf{g})$ be the unit sphere endowed with the standard Riemannian metric.     Katok metric is constructed as follows. Let $X$ be the vector field of rotation around the axis connecting the north and south poles of the  sphere such that $\mathbf{g}(X,X)<1$. In the 
the following spherical coordinate on $\mathbb{S}^2$,
\begin{align}
\label{eqn:spherical coordinate}
(\psi,\theta)\mapsto (\cos(\psi)\cos(\theta),\sin(\psi)\cos(\theta),\sin(\theta)).
\end{align}
We have
\begin{align*}
X=r\partial_{\psi},\ \Vert r\Vert<1.
\end{align*}
Now for any $p\in M$, the indicatrix $I_pM$ of the constructed Finsler function $F$ is obtained by shifting   the unit sphere $S_p\mathbb{S}^2$  of $\mathbf{g}$ (which is the indicatrix with respect to $\mathbf{g}$)  by $X$. That is,
\begin{align*}
I_pM\coloneqq \lbrace v+X_p\colon v\in T_p\mathbb{S}^2,\ \mathbf{g}(v,v)=1\rbrace
\end{align*}
This yields a well-defined Finsler metric as $\mathbf{g}(X,X)<1$. This  family of metrics depending on the parameter $r$ was constructed by A.~Katok  in \cite{Katok}. It is widely used in Finsler geometry and in the  theory of dynamical system as source of examples and counterexamples. It has constant flag curvature by \cite{bao2004zermelo,Foulon,Sh-1}, and by \cite{BFIMZ}, any metric of constant flag curvature on the 2-sphere has geodesic flow conjugate to that of a Katok metric.

As the measure $\nu_p$ we consider the Lebesgue measure as described in \S \ref{sec:average}; let us calculate the generator of the corresponding diffusion process $\xi$.  
  
By Theorem~\ref{thm:convergence main theorem}, the diffusion process $\xi$ generated by $\lbrace \nu_p\rbrace_{p\in M}$ has generator 
$A$ such that
\begin{equation}
\label{e:KatokA}
\begin{aligned}
Af(p)=& \di f(X)(p)+\dfrac{1}{2} \biggr\lbrace f_{ij}\int_{D_pM}(Y^i-X^i)(Y^j-X^j)\, \nu_p(\di Y) \\
& - f_k\int_{D_pM}\Gamma^k_{ij}\left(p,Y-X \right)(Y^i-X^i)(Y^j-X^j)\, \nu_p(\di Y) \biggr\rbrace\\
\end{aligned}
\end{equation}
Here $\Gamma^k_{ij}$ are the formal Christoffel symbols of the second kind of $(M,F)$ and $f\in \mathcal C^{\infty}$.

As $\nu_p$ is induced by a Lebesgue measure on $T_pM\simeq\mathbb{R}^m$, we also denote this Lebesgue measure by $\nu_p$ for simplicity. For any $p\in M$, the set 
\begin{align*}
\hat{D}_PM\coloneqq \lbrace Y\in T_pM\ \vert\ Y+X(p)\in D_pM\rbrace
\end{align*}
is just the closed unit ball on $T_p\mathbb{S}^2$ with respect to $\mathbf{g}$.
Since $\nu_p$ is translation invariant, Equation \eqref{e:KatokA} becomes
\begin{align}
\label{e:KatokB}
Af(p)=& \di f(X)(p)+\dfrac{1}{2} \biggr\lbrace f_{ij}\int_{\hat{D}_pM} Y^i Y^j \nu_p(\di Y) 
 - f_k\int_{\hat{D}_pM}\Gamma^k_{ij}\left(p,Y \right)Y^iY^j\, \nu_p(\di Y) \biggr\rbrace
\end{align}
Note the integrand in the equation above is second order homogeneous in $Y$. By Fubini theorem, for any $p\in M$ there is a finite measure $\eta_p$ on $S_p\mathbb{S}^2$ such that for any integrable second order homogeneous 
function $h$ on $T_pM$, we have
\begin{align}
\int_{\hat{D}_pM} h(Y)\nu_p (\di Y)=\int_{S_p\mathbb{S}^2} h(Y) \eta_p (\di Y)
\end{align}
Because $\nu_p$ is invariant under any orthogonal transformation on $T_p\mathbb{S}^2$ with respect to $\mathbf{g}$, it is clear $\eta_p$ is a multiple of the canonical angular measure $m_p$ on $S_p\mathbb{S}^2$ with respect to $\mathbf{g}$. A straight forward computation shows $\eta_p=\dfrac{1}{4\pi}m_p$. 
Hence from \eqref{e:KatokB}, we get
\begin{align}
Af(p)=\di f(X)(p)+\dfrac{1}{8\pi} \biggr\lbrace f_{ij}\int_{S_p\mathbb{S}^2} Y^i Y^j m_p(\di Y) 
 - f_k\int_{S_p\mathbb{S}^2}\Gamma^k_{ij}\left(p,Y \right)Y^iY^j\, m_p(\di Y) \biggr\rbrace
\end{align}
Let $\Delta$ be the Beltrami--Laplace operator of $\mathbf{g}$, and let $\hat{\Gamma}^k_{ij}$ be the Christoffel symbols of $\mathbf{g}$. A straightforward computation yields
\begin{equation}
\begin{aligned}
\dfrac{1}{8}\Delta f(p)=\dfrac{1}{8\pi}\biggr\lbrace f_{ij}\int_{S_p\mathbb{S}^2}Y^i Y^j \, m_p(\di Y)
- f_k\int_{S_p\mathbb{S}^2}\hat{\Gamma}^k_{ij}(p)  Y^i Y^j\, m_p(\di Y)\biggr\rbrace.
\end{aligned}
\end{equation}
This implies $\dfrac{1}{8}\Delta$ and $A$ have the same symbol. To compute the drift of $A$, we assume without loss of generality that $0\leq r<1$. First we have
\begin{equation}
\label{eqn:drift in coordiante}
\begin{aligned}
\left(A-\dfrac{1}{8}\Delta\right)f(p)&= \di f(X)(p)
+\dfrac{1}{8\pi}f_k \int_{S_p\mathbb{S}^2}\hat{\Gamma}^k_{ij}(p)  Y^i Y^j\, m_p(\di Y)\\
& -\dfrac{1}{8\pi} f_k \int_{S_p\mathbb{S}^2}\Gamma^k_{ij}\left(p,Y \right) Y^i Y^j \,m_p(\di Y). 
\end{aligned}
\end{equation}
Let $\Phi_t$ be the flow generated by $X$, we know from \cite[Theorem 1]{Foulon} that if $\gamma(t)$ is a geodesic of $(\mathbb{S}^2,\mathbf{g})$ with $\sqrt{\mathbf{g}(\dot{\gamma},\dot{\gamma})}=c$, then $\hat{\gamma}(t)=\Phi_{ct}\circ\gamma(t)$ is a geodesic of $F$ with initial vector $\hat{\gamma}'(0)=\dot{\gamma}(0)+cX(\gamma(0))$. But in the spherical coordinate given \eqref{eqn:spherical coordinate}, the flow $\Phi_t$ simply has the form
\begin{align*}
\Phi_t(\psi,\theta)=(\psi+rt,\theta).
\end{align*}
Then in this coordinate, we have
\begin{align*}
\dfrac{\di^2}{\di t^2}\biggr\vert_{t=0} \hat{\gamma}(t)= \dfrac{\di^2}{\di t^2}\biggr\vert_{t=0} \gamma(t).
\end{align*} 
By the geodesic equation, this is equivalent to 
\begin{align*}
\Gamma^k_{ij}(p,\hat{\gamma}'(0))(\hat{\gamma}'(0))^i (\hat{\gamma}'(0))^j=\hat{\Gamma}^k_{ij}(p)(\gamma'(0))^i(\gamma'(0))^j.
\end{align*}
Using this and \eqref{eqn:drift in coordiante}, a straight forward computation gives
\begin{align}
\Big(A-\frac{1}{8}\Delta\Big)=X+\dfrac{1}{4}r^2\cos(\theta)\sin(\theta)\cdot \dfrac{r^2\cos^2(\theta)-2}{\left( 1-r^2\cos^2(\theta)\right)^2}\cdot \partial_{\theta}.
\end{align}
This is the drift of the generator $A$. 
\begin{figure}
\begin{center}
 \includegraphics[width=0.45\textwidth]{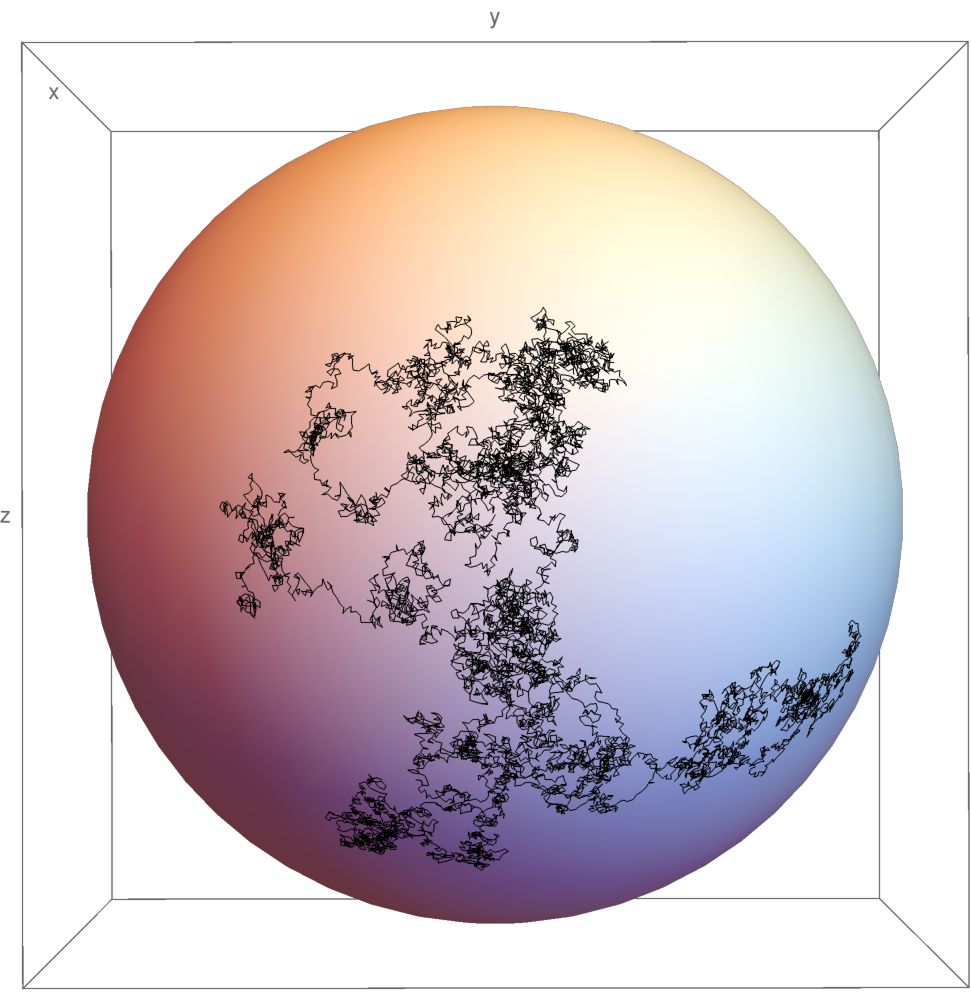} 
 \includegraphics[width=0.45\textwidth]{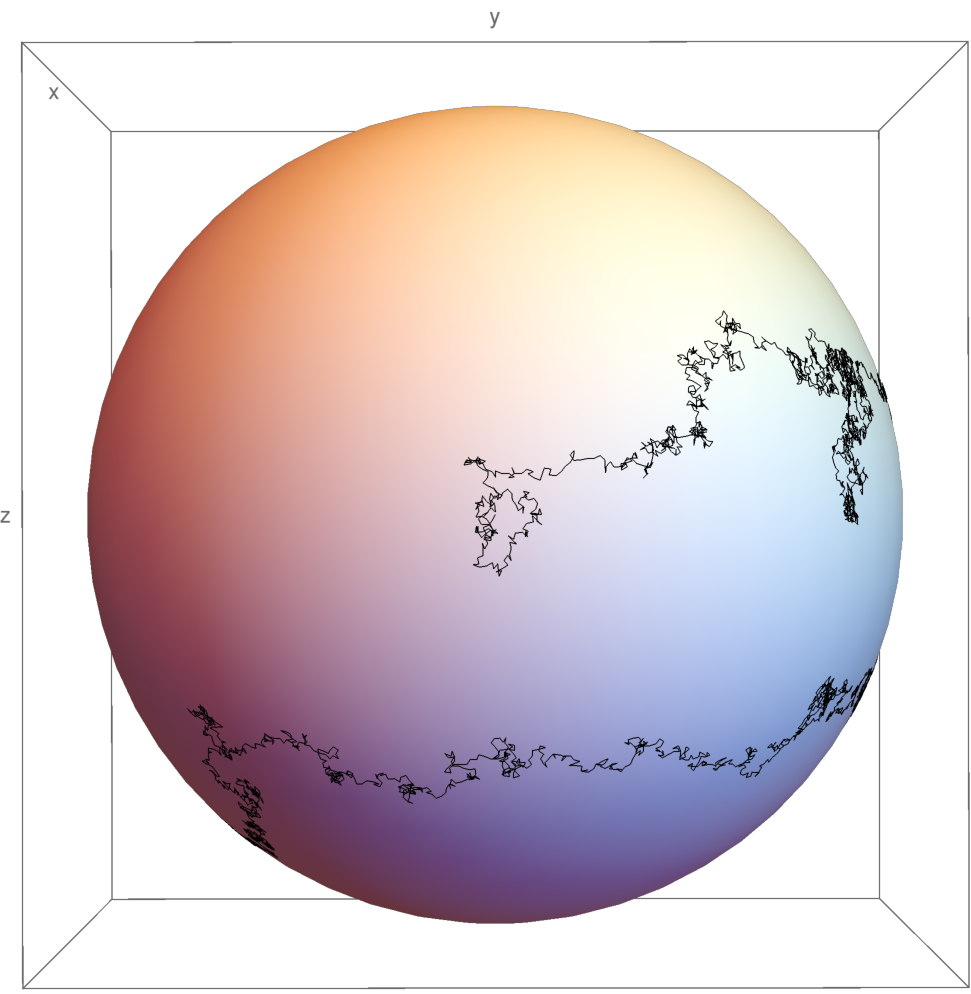} 
\end{center}
\caption{ (l.) A sample path of the Brownian motion on the standard sphere. (r.) A sample path of the 
diffusion with the generator $A$ given by \eqref{e:KatokA} with $r=1/2$.
\label{Fig:2}}
\end{figure}

On Figure \ref{Fig:2} one clearly sees the difference in the behaviours of the Brownian motion of the 
initial round metric on $\mathbb{S}^2$ and of the diffusion  process corresponding to 
the Katok metric with $r=1/2$ due to the drift given in \eqref{eqn:drift in coordiante}. 
Of course the pictures are just the pictures of the corresponding geodesic random walks with a sufficiently large $N$. 
Note that the same random seed was used in both pictures.

\section{Preliminaries.}
In this section, we give a short review of the tools in Finsler geometry  which will be used in our proof in later sections and formalise definitions of random 
geodesic walks which will allow us to apply the machinery from the theory of stochastic processes.

\subsection{Finsler geodesics and properties of bounded geometry.}\label{section:Finsler preilminary}

It is well-known that  stationary points  of the energy functional  \eqref{eq:F2} are solutions of the  Euler-Lagrange  equation  which in our situation is equivalent to the following      system of ODEs:  
\begin{align}
&\frac{\di x^i}{\di t}=y^i,\\
&\frac{\di y^k}{\di t}+\Gamma^k_{ij}(x,y)y^iy^j=0. 
\end{align}
Here $\Gamma^k_{ij}$ are the formal Christoffel symbols of the 2nd kind:
\begin{align} \label{gamma} 
\Gamma^k_{ij}(x,y)=\dfrac{1}{2} g^{ks}\left(\dfrac{\partial g_{is}}{\partial x_j}+\dfrac{\partial g_{js}}{\partial x_i}-\dfrac{\partial g_{ij}}{\partial x_s}\right)(x,y),\quad y\neq 0
\end{align}
It is immediate from \eqref{gamma} that for any $\lambda>0$ and $y\neq 0$, we have 
\begin{align}
\Gamma^k_{ij}(x,y)=\Gamma^k_{ij}(x,\lambda y).
\end{align}
Denote the class of real-valued $k$-times continuously differentiable real-valued functions on $M$ with compact support by $\mathcal{C}^k_K$. Suppose $f\in \mathcal{C}^{k_0}_K$ with compact support $K_f$. For any $Y\in TM$, define $f_Y(t)=f\circ \gamma_Y(t)$, where $\gamma_Y(t)$ is the geodesic with initial vector $Y$ as before. 
\begin{lem}
\label{lemma:norm}
Suppose that $f\in \mathcal{C}^{k_0}_K$. There exists some constant $c$ such that 
\ba
\biggr\lvert\left(\dfrac{\di ^k}{\di t^k}f_Y\right)(t)\biggr\rvert\leq c F^k(Y),\quad \forall k\leq k_0,
\ea
wherever it is well defined.
\end{lem}
\begin{proof}
First we show there is some constant $c$ such that $\biggr\lvert\left(\dfrac{\di ^k}{\di t^k}f_Y\right)(0)\biggr\rvert\leq c F^k(Y)$. Let $\tilde{X}$ be the geodesic spray on $TM$, and denote $\pi\colon TM\rightarrow M$ the canonical projection. The function $\pi^*f$ is $\mathcal{C}^{k_0}$, so $(\mathcal{L})^k_{\tilde{X}}(\pi^*f)$ is continuous on $TM\setminus \lbrace 0\rbrace$ for any $k\leq k_0$. Hence for some constant $c$, we have 
$$\lvert(\mathcal{L})^k_{\tilde{X}}(\pi^*f)(\hat{p})\rvert\leq c ,\ \forall\hat{p}\in IK_f,\ \forall k\leq k_0. $$
For any $p\in M$ and $Y\in T_pM$, we have $$\left(\dfrac{\di ^k}{\di t^k}f_Y\right)(0)=(\mathcal{L})^k_{\tilde{X}}(\pi^*f)(p,Y).$$
In addition, for $Y\neq 0$, let $Y'=\dfrac{Y}{F(Y)}$. Using $$f_Y(t)=f\circ \gamma_{F(Y)Y'}(t)=f_{Y'}(F(Y)t),$$
we get 
$$\biggr\lvert\left(\dfrac{\di ^k}{\di t^k}f_Y\right)(0)\biggr\rvert=\biggr\lvert\left(\dfrac{\di ^k}{\di t^k}f_{Y'}\right)(0)\biggr\rvert F^k(Y).$$
Then for $p\in K_f$ and $Y\neq 0$, we have
$$\biggr\lvert\left(\dfrac{\di^k}{\di t^k}f_Y\right)(0)\biggr\rvert=(\mathcal{L})^k_{\tilde{X}}(\pi^*f)(p,Y')\cdot F^k(Y)\leq c F^k(Y).$$
For $p\notin K_f$, we have $f$ vanishes identically on some neighbourhood of $p$. Then the function $t\mapsto f_Y(t)$ is constant near $t=0$. For $Y=0$, then function $f_Y(t)$ is also constant. It follows that $\biggr\lvert\left(\dfrac{\di ^k}{\di t^k}f_Y\right)(0)\biggr\rvert\leq c F^k(Y)$.

Next, given any geodesic $\gamma_Y$, let $Y(t)$ be its velocity field. We have  $$F(Y(t))=F(Y(0))=F(Y),$$
$$
\biggr\lvert\left(\dfrac{\di ^k}{\di t^k}f_Y\right)(t)\biggr\rvert=\biggr\lvert \left( \dfrac{\di ^k}{\di t^k}f_{Y(t)}\right)(0)\biggr\rvert\leq cF^k(Y(t))=cF^k(Y).
$$
This completes the proof.
\end{proof}

The injective radius at $p$ is defined by
\begin{align*}
\mathrm{inj}_M(p)& \coloneqq \inf\lbrace r>0\colon \exp_p\vert_{D_p(r)}\ \mathrm{is\ injective} \rbrace\\
\mathrm{inj}_M & \coloneqq \inf\lbrace \mathrm{inj}_M(p)\colon p\in M \rbrace
\end{align*}
The conjugate radius is defined similarly by
\begin{align*}
\mathrm{con}_M(p)& \coloneqq \inf\lbrace r>0\colon \exp_p\vert(D_p(r))\ \mathrm{is\ an\ immersion}\rbrace\\
\mathrm{con}_M & \coloneqq \inf\lbrace \mathrm{con}_M(p)\colon p\in M\rbrace.
\end{align*}
Clearly we always have $\mathrm{inj}_M(p)\leq \mathrm{con}_M(p)$ for any $p\in M$. The conjugate radius and flag curvature are related by the following well known result, \cite[Section 9.5]{bao2000introduction}.
\begin{prp}[Morse--Schoenberg]
\label{prop:Morse}
Suppose $(M,F)$ is a Finsler manifold such that its flag curvature $\Vert K\Vert \leq \lambda$. Then the conjugate radius is bounded from below by $\mathrm{con}_M\geq \dfrac{\pi}{\sqrt{\lambda}}$, hence strictly positive.
\end{prp}

\subsection{Formal definition of rescaled geodesic random walks\label{sec:form1}}
We begin this section by a brief review of the basic definitions in Markov processes used in this paper. Roughly speaking, a stochastic process is said to be Markovian if its future states depend only upon the present state, regardless of its past state.

\begin{dfn}
 Let $(\Omega,\mathcal{F},\mathbf{P})$ be a probability space.
 An $M$-valued process 
 $\xi\colon \Omega\times[0,\infty)\rightarrow M$ is Markovian if for each Borel subset of $B$ of $M$,
for all $n\geq 1$, $0\leq s_1<\cdots<s_n<s<t$ 
 we have,
\ba
\mathbf{P}(\xi_t\in B\vert \xi_{s_1},\dots, \xi_{s_n},\xi_s)=\mathbf{P}(\xi_t\in B\vert \xi_s).
\ea
\end{dfn}

The transition probability function for a Markov process is defined by
\begin{align*}
P(p,s,t,B)=\mathbf{P}(\xi_t\in B\vert \xi_s=p),\ \forall p\in M,\ \forall 0\leq s\leq t.
\end{align*}
We say $(\xi_t)_{t\geq 0}$ is time homogeneous if the following holds.
\begin{align*}
P(p,s,t,B)=P(p,0,t-s,B),\ \forall p\in M,\ \forall 0\leq s\leq t.
\end{align*}
All Markov processes considered in this paper are time-homogeneous. A time homogenous Markov process $\xi$ defines 
a semigroup $T=(T_t)_{t\geq 0}$ of linear operators on the measurable functions 
 $\mathcal{B}$
on $M$ by
\begin{align*}
T_t(f)(p)=\mathbf{E}_p[f(\xi_t)],\ p\in M,\ t\geq 0.
\end{align*} 
We say a Markov process $\xi$ is {\it Feller} if the
semigroup $T$
is a strongly continuous semigroup of positive contractions on the Banach space $\mathcal{C}_0$.

In the introduction, we gave a slightly informal definition of (rescaled) geodesic random walks. We now give a formal definition. 

Let $(M,F)$ be a geodesically complete Finsler manifold. Let $\{ \nu_p\}_{p\in M}$ be a family of measures 
such that each $\nu_p$ is a probability measure on $T_pM$. 
Denote by $\mu_p$ the mean of $\nu_p$
\begin{align*}
\mu_p\coloneqq\int_{T_pM} Y\, \nu_p (\di Y).
\end{align*}
In our setting (Hypothesis $\mathbf{H}_\nu$), the probability measures are compactly supported so $\mu_p$ exists and is finite.

\begin{dfn}
Let $N\geq 1$ and let $p_0\in M$ be fixed. A random process $(\zeta_k^N, Y_{k+1}^N)_{k\geq 0}$ is called a (rescaled) 
discrete time geodesic random walk on $M$ with initial point $p_0$ and with increments 
$\{Y_{k+1}^N\}_{k\geq 0}$ compatible with the family $\{\nu_p\}_{p\in M}$ if
\begin{enumerate}
\item the process $\zeta_k^N$ is $M$-valued, and $Y_{k+1}^N$ is $T_{\xi_k^N}M$-valued,
\item $\zeta_0^N=p_0$,
\item for each $k\geq 0$, $\text{Law}(Y_{k+1}^N)=\nu_{\zeta_k^N}^N$ where
\ba
\nu_p^N(B)&=\int_{T_pM} \mathbb{I}_B\Big( \frac{Y-\mu_p}{\sqrt N}+\frac{\mu_p}{N}  \Big)\,\nu_p(\di Y)
\ea
for any measurable $B\subseteq T_pM$,
\item $\zeta_{k+1}^N=\exp_{\zeta_k^N}(Y_{k+1}^N)$, $k\geq 0$.
\end{enumerate}
\end{dfn}
Hence the processes $\zeta^N$ are defined by the family of measures $\{\nu_p\}_{p\in M}$ and the geometry of the exponential mapping $\exp_p$. 
The random walks are time-homogeneous since the family $\{\nu_p\}_{p\in M}$ does not depend on $k$.

In the classical (Euclidean) setting, random walks are processes with independent increments. 
In our setting the independence is understood 
in the conditional sense, i.e.\ the increments $Y_{k+1}^N$ depend only on the current position $\zeta^N_k$ and not on the 
previous positions and increments. More precisely, we introduce the natural filtration
 \begin{align}
\mathcal F_k^N:=\sigma\{(\zeta_0^N, Y^N_{1}),\dots, (\zeta^N_{k-1}, Y_k^N)\},\quad k\geq 1,
\end{align}
and say that the increments of $\zeta^N$ are independent if for each 
$f\in \mathcal{C}_b(\oplus_{i=0}^{k+1} M,\mathbb R)$
\begin{align}
\mathbf E \Big[f(\zeta_0^N,\dots,\zeta_{k}^N,\zeta_{k+1}^N)\Big|\mathcal F_{k}^N\Big]
= \int_{T_{\zeta_{k}^N}M} f(\zeta_0^N,\dots,\zeta_{k}^N, 
\exp_{\zeta_{k}^N}(Y))\,\nu_{\zeta_{k}^N}^N(\di Y).
\end{align}
It is clear that $\zeta^N$ is a homogeneous discrete time $M$-valued Markov chain
with the one-step transition operator
\begin{align}
P^N f(p) = \mathbf{E}_p f(\zeta^N_1) 
=\int_{T_{p}M} f\Big(\exp_p\Big( \frac{Y-\mu_p}{\sqrt N}+\frac{\mu_p}{N}\Big)\Big)\,\nu_{p}(\di Y),\quad f\in \mathcal{C}_b(M,\mathbb R).
\end{align}

Since we work in a continuous time setting it is convenient to transform the discrete time Markov chain $\zeta^N$ into a 
continuous time Markov process. This can be done by a standard subordination procedure.

Let $Q=(Q_t)_{t\geq 0}$ be a standard Poisson process independent of $\{\zeta^N\}$. Define a \emph{pseudo-Poisson}  process 
\begin{align*}
\xi^N_t= \zeta^N_{Q_{Nt}} ,\quad t\geq 0.
\end{align*}

Note that the sample paths of $\xi^N$ belong to $D([0,\infty),M)$. 
Hence, the Markov processes $\xi^N$ induce probability distributions $\mathbf P^N$ on the 
path space $D([0,\infty),M)$. It is easy to see that the transition semigroup 
$T^N=(T^N_t)_{t\geq 0}$ of $\xi^N_t$ has the form
\begin{align}
\label{eqn:semigroup}
T^N_t(f)(p)=\mathbf{E}_p[f(\xi^N_t)]=\ex^{-Nt}\sum_{k=0}^{\infty}\dfrac{(Nt)^k}{k!}(P^N)^k(f)(p),\ f\in\mathcal{B}.
\end{align}
Finally we introduce a of family continuous $M$-valued processes defined by
\begin{align*}
\hat \xi_t^N&=\exp_{\zeta^N_{k}}\Big(N\Big(t-\frac{k}{N}\Big)Y^N_{k+1} \Big) ,\quad t\in \Big[\frac{k}{N},\frac{k+1}{N}\Big],\quad k\geq 0.
\end{align*}

Since the manifold is geodesically complete, the processes $\zeta^N$, $\xi^N$ and $\hat \xi^N$ are well-defined 
for each $N\geq 1$. By construction the processes $\hat \xi^N$ have piecewise smooth sample paths consisting of geodesic segments, and 
induce probability distributions $\hat{\mathbf P}^N$ on the 
path space $C([0,\infty),M)$ of continuous $M$-valued functions. These are the geodesic random walks introduced and discussed in the Introduction, see   Fig.~\ref{fig:0}  there, 
and in \S \ref{sec:Katok}, see Fig.~\ref{Fig:2}. Although the processes $\hat \xi^N$ are not Markovian, the convergence of the continuous time processes
$(\zeta^N_{[Nt]})_{t\geq 0}$, $(\xi^N_t)_{t\geq 0}$ and $(\hat \xi^N_t)_{t\geq 0}$ is equivalent, see, e.g.\ \cite[Theorem 17.28]{Kallenberg-02}.
In the next sections we will mainly work with the Markov processes $\xi^N$.

\section{Proof of the main theorem.}
In this section we prove the convergence of the  geodesic random walks $\lbrace \xi^N\rbrace$.  
We will assume that the Finsler manifold $(M, F)$ is forward complete and connected (Hypothesis \textbf{H}$_{c}$) 
and has bounded geometry (Hypothesis  \textbf{H}$_b$)
and that the measures $\lbrace \nu_p\rbrace_{p\in M}$ satisfy the condition \textbf{H}$_{\nu}$ from  \S \ref{sec:main}.

\subsection{Generators  of geodesic random walks.}

In this section we show the $N$-scaled geodesic random walks on a complete Finsler manifold $(M,F)$ with bounded geometry are Feller. 
\begin{lem}
\label{lemma:smooth}
Let $\mathbf{H}_\nu$ hold true and $k\geq 0$. For any $\mathcal C^k$-smooth function $f\colon TM\to \mathbb R$ 
the mapping
\ba
p\to \int_{T_pM} f(Y)\,\nu_p(\di Y)
\ea
is also $\mathcal C^k$-smooth.
\end{lem}
This lemma is obvious since each $\nu_p$ is only supported on $D_pM$. Indeed, integral over a compact set of a function smoothly depending on parameters smoothly depends on the parameters.  

Now we are ready to show the semigroups $\lbrace T^N\rbrace$ are Feller and give the formula of the generators.

\begin{prp}
\label{prop:N-walk generator}
Suppose $(M,F)$ is complete and uniform elliptic. In addition, assume the measures $\lbrace \nu_p\rbrace$ satisfy the hypothesis $\mathbf{H}_\nu$. 
Then for each $N\geq 1$, the family of operators $T^N=( T^N_t)_{t\geq 0}$ is a conservative Feller semigroup with the generator
\begin{align}
\label{e:AN}
A_N f =N\Big( P^N f -f\Big),\quad  f\in \mathcal{C}_0.
\end{align}
\begin{proof}
Let $N\geq 1$ be fixed. 
Since by constriction $T^N$ is a strongly continuous semigroup of a pseudo-Poisson process, its generator has the form \eqref{e:AN} by Theorem 19.2 
from \cite{Kallenberg-02}. It is conservative due to assumption $\mathbf{H}_{c}$.
Let us show that $T^N_t$ maps $\mathcal \mathcal{C}_0$ into itself for $t\geq 0$. Since we have
\begin{align}
\Vert P^N f \Vert\leq \Vert f\Vert,
\end{align}
the series in \eqref{eqn:semigroup}
converges uniformly. It suffices to show $P^N$ maps $\mathcal{C}_0$ into itself.

By Lemma \ref{lemma:smooth} the mean value $\mu_p$ is a $\mathcal C^\infty$ vector field. 
Since the exponential map for Finsler manifold is at least $\mathcal C^1$, then
$P^N$ maps continuous functions into continuous functions.

For any $\varepsilon>0$, choose some compact $K\subseteq M$ such that $\vert f(x)\vert<\dfrac{\varepsilon}{2}$ for $x\notin K$. 
Fix any $p_0\in K$, and define the closed forward balls at $p_0$ for $R\geq 0$ by
\begin{align*}
B^+_{p_0}(R)\coloneqq \lbrace q\in M\colon d_a(p_0,q)\leq R\rbrace.
\end{align*}
Because $K$ is compact, there exists some $R_0>0$ such that $K\subset B^+_{p_0}(R)$ for all $R\geq R_0$. By the 
Hopf--Rinow theorem (see Theorem 6.6.1 of \cite{bao2000introduction}), the forward closed balls $B^+_{p_0}(R)$ are also compact.

For any $0\neq Y\in T_pM$ and $p\in M$, the uniform ellipticity condition in Definition~\ref{definition:bounded geometry} gives
\begin{align}
\label{eqn:symmetry coefficient}
F^2(-Y)=g_{-Y}(Y,Y)\leq C^2 g_Y(Y,Y)=C^2F^2(Y).
\end{align}
It follows that 
\begin{align}
d_a(q,p_0)& \leq C d_a(p_0,q)\leq CR_0,\ \forall q\in K;\\
d_a(p,p_0) & \geq \dfrac{1}{C} d_a(p_0,p)\geq \dfrac{R}{C},\ \forall p\in (B^+_{p_0}(R))^c,\ \forall R\geq 0.
\end{align}
Let $R_1\coloneqq C(C+2+CR_0)$, then $\forall p\in (B^+_{p_0}(R_1))^c$ and $\forall q\in K$, we have
\begin{align}
d_a(p,q)\geq d_a(p,p_0)-d_a(q,p_0)\geq \dfrac{R_1}{C}-CR_0>C+1
\end{align}
On the other hand, for $p\in M$ and $Y\in D_pM$, we have
\begin{align}
d_a(p,\textbf{e}_p^N(Y))\leq F(\dfrac{1}{\sqrt{N}}(Y-(1-1/\sqrt{N})\mu_p))\leq F(Y)+F(-\mu_p)\leq C+1.
\end{align}
Hence, $\forall Y\in D_pM$ and $p\in (B^+_p(R_1))^c$, we have $\textbf{e}_p^N(Y)\notin K$.
It follows that $\forall p\in (B^+_p(R_1))^c$:
\begin{align}
\vert P^N f (p)\vert=\biggr\lvert \int_{T_pM} f\circ \textbf{e}_p^N(Y)\, \nu_p (\di Y) \biggr\rvert \leq \dfrac{\varepsilon}{2}.
\end{align}
That is to say $\lVert P^N f \rVert\leq \dfrac{\varepsilon}{2}$ outside the compact set $B^+_P(R_1)$. We conclude that $P^N f \in \mathcal C_0$.
This completes the proof.
\end{proof}
\end{prp}

\subsection{Convergence of the generators of geodesic random walks.}
In this section, we prove the generators $A_N$ converge on the space $\mathcal C^{\infty}_K$ to some second order elliptic operator with smooth coefficients.

Denote
\ba
f_{Y-\mu_p}(t)=f\circ \gamma_{Y-\mu_p}(t)
\ea

\begin{prp}
\label{prop:limit generator}
Let $A$ be the  differential operator defined by
\begin{align}
Af(p)\coloneqq \di f(\mu_p)+\dfrac{1}{2}\int_{T_pM} \dfrac{\di^2}{\di t^2}\biggr\vert_{t=0} f_{Y-\mu_p}(t)\,  \nu_p (\di Y),\  f\in \mathcal C^{2}.
\end{align}
Then $A$ is a second order positive definite elliptic operator of smooth coefficients and for each 
$f\in \mathcal C^{\infty}_K$
\begin{align}
\lim_{N \to \infty}\lVert A_Nf-Af\|=0.
\end{align}
\end{prp}
\begin{proof}
The proof follows the steps from \cite{Jorgensen} in the Riemannian case. By computing the Taylor expansion 
of $A_N f$, we show the convergence of the first and second order terms and vanishing of other higher order terms as $N\rightarrow \infty$.

Take any $f\in \mathcal C^{\infty}_K$. We have
\ba
A_N(f)(p)& =N\Big( P^N(f)(p)-f(p)\Big)\\
 & =  N\ \int_{T_pM} \Big [ f\circ \gamma_{Y-(1-1/\sqrt{N}) \mu_p}\Big(\dfrac{1}{\sqrt{N}}\Big)-f(p)\Big]\, \nu_p (\di Y) 
\ea
Then for any $p\in M$ and $Y\in D_pM$, the Taylor expansion of 
\ba
f_{Y-(1-1/\sqrt{N}) \mu_p }(t)=f\circ \gamma_{Y-(1-1/\sqrt{N}) \mu_p}(t)
\ea
gives
\begin{align*}
f_{Y-(1-1/\sqrt{N}) \mu_p}\left(\dfrac{1}{\sqrt{N}}\right)= & f(p)+ \dfrac{1}{\sqrt{N}} \di f_p(Y-(1-1/\sqrt{N}) \mu_p)\\
& + \dfrac{1}{2N}\dfrac{\di^2}{\di t^2}\biggr\vert_{t=0}\left(f_{Y-(1-1/\sqrt{N}) \mu_p}(t)\right)+R_N(p,Y).
\end{align*}
Thus we have
\ba
\label{eqn:AN taylor}
A_Nf (p) & = N \int_{T_pM}\biggr\lbrace  \dfrac{1}{\sqrt{N}} \di f_p(Y-(1-1/\sqrt{N}) \mu_p) \\
& + \dfrac{1}{2N}\dfrac{\di^2}{\di t^2}\biggr\vert_{t=0}(f_{Y-(1-1/\sqrt{N}) \mu_p}(t)) +R_N(p,Y)\biggr\rbrace\, \nu_p (\di Y)\\
 &= \di f(\mu_p)+\int_{T_pM} \dfrac{1}{2}\dfrac{\di ^2}{\di t^2} \biggr\vert_{t=0}(f_{Y-(1-1/\sqrt{N}) \mu_p}(t))\,\nu_p (\di Y)
 +\int_{T_pM}N R_N(p,Y)\, \nu_p (\di Y).
\ea
Using Lemma~\ref{lemma:norm} and equation~\eqref{eqn:symmetry coefficient}, for any $Y\in D_pM$ and $p\in M$, there is some constant $c_f>0$ such that
\ba
\label{eqn:remainder bound}
| R_N(p,Y)| & \leq \dfrac{1}{N\sqrt{N}}\ \sup_{t \in[0,1/\sqrt{N}]}\ 
\biggr\lvert \dfrac{\di^3}{\di t^3} f_{Y-(1-1/\sqrt{N}) \mu_p}(t) \biggr\rvert \\
& \leq \dfrac{c_f}{N\sqrt{N}} F^3(Y-(1-1/\sqrt{N}) \mu_p)\\
& \leq \dfrac{c_f}{N\sqrt{N}}  (F(Y)+F(-\mu_p))^3 \\
& \leq  \dfrac{c_f}{N\sqrt{N}}(C+1)^3.
\ea
Clearly, we have from \eqref{eqn:remainder bound}:
\begin{align}
\lim_{N\to \infty} \sup_{p\in M, Y\in D_pM}| N R_N(p,Y)|=0.
\end{align}
The last term in \eqref{eqn:AN taylor} tends to zero, since $\nu_p$ is only supported on $D_pM$.

For the second order term, in a canonical coordinate of $TM$, we have
\ba
\label{eqn:second order term} 
& \dfrac{\di ^2}{\di t^2} \biggr\vert_{t=0}(f_{Y-(1-1/\sqrt{N}) \mu_p}(t))\\
 &= f_{ij}\cdot y^i\left(Y-(1-\dfrac{1}{\sqrt{N}} )\mu_p\right)  y^j\left(Y-(1-\dfrac{1}{\sqrt{N}} )\mu_p\right)\\
&\quad - f_k\cdot \Gamma^k_{ij}\left(p,Y-(1-\dfrac{1}{\sqrt{N}})\mu_p  \right) y^i\left(Y-(1-\dfrac{1}{\sqrt{N}} )\mu_p\right)  y^j\left(Y-(1-\dfrac{1}{\sqrt{N}} )\mu_p\right).
\ea
Since the formal Christoffel symbols are bounded on each compact local coordinate, the right-hand-side of \eqref{eqn:second order term} converges to
\begin{align*}
f_{ij}\cdot y^i(Y-\mu_p)y^j(Y-\mu_p) - f_k \cdot \Gamma^k_{ij}(p,Y-\mu_p)y^i(Y-\mu_p)y^j(Y-\mu_p),
\end{align*} 
as $N\rightarrow \infty$ uniformly on  $DK$  for each compact chart $K\subseteq M$, where $DK=\lbrace Y\in D_pM\colon p\in K\rbrace$.

Choose a smooth coordinate on some open $U\subseteq M$. The chain rule implies $A$ has the following form in this coordinate.
\begin{align*}
Af(p) =& \di f(\mu_p)+ \dfrac{1}{2} \biggr ( f_{ij}\int_{T_pM} y^i(Y-\mu_p)y^j(Y-\mu_p)\,\nu_p (\di Y)\\
 & - f_k\int_{T_pM} \Gamma^k_{ij}(p,Y-\mu_p)y^i(Y-\mu_p)y^j(Y-\mu_p)\, \nu_p   (\di Y) \biggr ).
\end{align*}
Because $f$ has compact support, we have
\begin{align}
\lim_{n\rightarrow \infty}\lVert A_N f-A f \rVert=0,\quad f\in \mathcal C^{\infty}_K.
\end{align}
It follows that the symbol of $A$ is: 
\begin{align}
\sigma(A)(p)=\dfrac{1}{2}\int_{T_pM} \otimes^2 (Y-\mu_p)\, \nu_p (\di Y).
\end{align}
For each $p\in M$, the measure $\nu_p$ is induced by either a smooth non-zero $(m-1)$-form on $I_pM$ or an $m$-form 
$T_pM$ (condition $\mathbf{H}_{\nu}$). Then $\sigma(A)$ is strictly positive definite, and hence $A$ is a
strictly elliptic operator.

In each compact local coordinate, the functions $\lbrace\Gamma^k_{ij}\rbrace$ are bounded, and smooth on $TM\setminus \lbrace 0\rbrace$. 
It follows that $A$ has smooth coefficients. This completes the proof.
 
%
\end{proof}

\subsection{Tightness of the family $\lbrace \xi^N\rbrace$.}
In this section we prove the family of random walks 
$\lbrace \xi^N\rbrace$ is tight in $D([0,\infty),M)$. Recall that the symmetrized distance $d$ makes $(M,d)$ a complete separable metric space, as $(M,F)$ is forward complete and has bounded geometry.

\begin{prp}
\label{prop:simply connected case tightnesss}
Let the Finsler manifold  $(M,F)$ and the family 
$\{ \nu_p\}_{p\in M}$ 
satisfy Assumptions  \emph{\textbf{H}}$_{c}$,  \emph{\textbf{H}}$_b$ and $\mathbf{H}_\nu$. 
Then the family of random walks $\lbrace \xi^N\rbrace_{N\geq 1}$ is tight in $D([0,\infty),M)$.
\end{prp}
\begin{proof}
The statement follows from the Aldous criteria (Lemma \ref{l:AC}) and the compact containment condition (Lemma \eqref{l:cc}) that will be proven 
in this section.
\end{proof}

Our goal consists of obtaining uniform estimates for the oscillation of the random walks $\xi^N$, see Equation\ \eqref{eqn:Aldous}.
Because $M$ is in general non-compact, the injective radius lower bound $\operatorname{inj}_M$ can be zero. 
Since the non-symmetrized distance function 
$d_a(p,\cdot)$
is smooth only within the injective radius, we work on the tangent bundle $TM$ to bypass this technical problem. 
To prepare the proof of Lemma \ref{l:AC} as well as Lemma~\ref{l:cc}, for $R\geq 0$, define 
\begin{align*}
D_p(R)\coloneqq \lbrace Y\in T_pM\colon F(Y)\leq R\rbrace.
\end{align*}
We make the following construction.

The condition $K\leq \lambda$ implies there exists some $0<\delta_c<1$ so that the conjugate radius $\mathrm{con}_M>\delta_c$, see Proposition \ref{prop:Morse}. 
For each $p\in M$, the exponential map $\exp_p$ is a smooth immersion on $D_p(\delta_c)$ except possibly at $0$. 
Then we can construct a 
geodesically complete smooth Finsler function $F_p$ on $T_pM$ such that $F_p=(\exp_p)^*(F)$ on $D_p(\frac{\delta_c}{2})$, 
while $F_p$ is the standard Minkowski metric on 
$T_pM\setminus D_p(1)$, under any standard identification $T_pM\simeq\mathbb{R}^m$. To distinguish it from the distance functions on $(M,F)$, we denote the asymmetric and symmetric 
distance on $(T_pM,F_p)$ by $d_a^p$ and $d^p$, respectively.
Note that the injective radius of $F_p$ at $0\in T_pM$ is at least $\frac{\delta_c}{4}$.

 Now for each $p\in M$ we construct 
the measures $\lbrace\tilde{\nu}_q\rbrace_{q\in T_pM}$, 
so that on $D_p(\delta_c/2)$, the measures $\lbrace\tilde{\nu_q}\rbrace_{q\in D_p(\delta_c/2)}$ 
are the lift of $\lbrace \nu_o\rbrace_{o\in M}$ by the exponential map $\exp_p$. In addition, we require the measures 
$\lbrace \tilde{\nu}_q\rbrace_{q\in T_pM}$ satisfy the condition $\mathbf{H}_{\nu}$.

Then for each $p\in M$ and $N\geq 1$, we construct an $N$-scaled geodesic random walk $\xi^{N,p}$ on the Finsler manifold $(T_pM,F_p)$ 
starting at $0\in T_pM$, using the prescribed measures $\lbrace \tilde{\nu}_q\rbrace_{q\in T_pM}$ as in 
Section~\ref{sec:form1}. Note as $(T_pM,F_p)$ satisfies $\mathbf{H}_{b}$ and $\mathbf{H}_{c}$, all results we proved earlier are true for the random walks $\xi^{N,p}$.
\begin{lem}
\label{lemma:generator bounds}
There exists some $\delta_0>0$ so that for each $\delta\in(0,\delta_0)$, 
there exists a family of functions $\lbrace f_p^{\delta}\rbrace_{p\in M}$ such that
\begin{enumerate}
\item \label{condition:local1}
Each $f_p^{\delta}$ is a function on $T_pM$ with $0\leq f_p^{\delta}\leq 1$ such that $f_p^{\delta}(0)=1$ and $f^{\delta}_p(q)=0$ if $q\notin D_p(\delta)$.
\item \label{condtion:op-bound}
Denote $A_{N,p}$ the generator associated to $\xi^{N,p}$, then there exists a constant $\tilde{C}(\delta)>0$ such that
\begin{align}
\sup_{N\geq 1}\sup_{p\in M} \sup_{q\in T_pM}\Big| A_{N,p}f_p^{\delta}(q)\Big| \leq \tilde{C}(\delta).
\end{align}
\end{enumerate}
\end{lem}
\begin{proof}
The general scheme to prove this lemma is as follows. We construct the family of functions $\lbrace f^{\delta}_p\rbrace$ using the distance functions $d_a^p(0,\cdot)$ on $(T_pM,F_p)$. 
The Hessian comparison theorem from \cite[Section 15.1]{shen2001lectures} applied to Finsler manifolds with bounded flag and $T$-curvature suggests the distance functions have uniformly bounded Hessians. This fact applied to 
the Taylor expansion of $f^{\delta}_p$ 
implies the family of functions we constructed satisfies the conditions listed in the Lemma.

For $\delta_c\in(0,1)$ chosen above, $C>0$ and $\lambda>0$ from Definition~\ref{definition:bounded geometry}, let
$0<\delta_0<\min\{\frac{\delta_c}{4(C+1)},\frac{\pi}{2\sqrt{\lambda}}\}$. Fix any $p\in M$, and denote $d_a^p(0,\cdot)$ the distance function  for $(T_pM,F_p)$ from $0\in T_pM$. Because the injective radius for $F_p$ at $0$ is at least $\frac{\delta_c}{4}$, the distance function $d_a^p(0,\cdot)$ is smooth on the open set $D_p(\delta)\setminus \lbrace 0\rbrace$ for each $0<\delta<\delta_0$.

Fix any $\delta\in (0,\delta_0)$. Let $\psi^{\delta}\colon \mathbb R\to\mathbb R$ be a smooth function with compact support contained in $[-\frac{\delta}{2},\frac{\delta}{2}]$. 
Further suppose $0\leq \psi^{\delta}\leq 1$ and $\psi^{\delta}\equiv 1$ on $I_1=[-\frac{\delta}{4},\frac{\delta}{4}]$. 
For each $p\in M$ , the function
\begin{align}
f_p^{\delta}(q):=\psi^{\delta} \circ  d_a^p(0,q),\quad q\in T_pM.
\end{align}
is smooth on $T_pM$ and satisfies condition \ref{condition:local1}.

To prove condition \ref{condtion:op-bound}, for any $q\in (T_pM,F_p)$, let $\tilde{\mu}_q$ be the mean of $\tilde{\nu}_q$. An argument similar to Proposition \ref{prop:N-walk generator} shows that for any $q\in T_pM$ 
\begin{align}
\label{eqn:operatorform1}
A_{N,p}f_p^{\delta} (q)=N\Big[ \int_{T_q(T_pM)} f_p^{\delta}\circ \gamma_{Y-(1-1/\sqrt{N})\tilde{\mu}_q}
\Big(\frac{1}{\sqrt{N}}\Big)\, \tilde{\nu}_q (\di Y) -f_p^{\delta}(q) \Big].
\end{align}
Note $\lbrace \tilde{\nu}_q\rbrace_{q\in T_pM}$ satisfies $\mathbf{H}_{\nu}$, so we only need to integrate over
$Y\in T_q(T_pM)$ with $F_p(Y)\leq 1$.

To simplify the notations, define for $Y\in T_q(T_pM)$
\begin{align}
h_N(Y)(t)& :=d_a\Big(0,\gamma_{Y-(1-1/\sqrt{N})\mu_q}(t)\Big),\\
h_N^{\delta}(Y)(t)& :=\psi^{\delta}\circ h_N(Y)(t)=f_p^{\delta}\circ \gamma_{Y-(1-1/\sqrt{N})\mu_q}(t), \quad t\geq 0.
\end{align}
By Taylor theorem, there exist functions $\lbrace t_N\rbrace$ with 
\ba
t_N\colon T_q (T_pM)\rightarrow \Big(0,\frac{1}{\sqrt{N}}\Big)
\ea
such that 
\begin{align}
\label{eqn:op-taylor}
h_N^{\delta}(Y)\Big(\frac{1}{\sqrt{N}}\Big)
=f_p^{\delta}(q)+\di f_p^{\delta}(q)\Big(\frac{Y-\tilde{\mu}_q}{\sqrt{N}}+\frac{1}{N}\tilde{\mu}_q\Big)
+\frac{1}{2N}\frac{\di ^2}{\di t^2}\biggr\vert_{t=t_N(Y)}\left(h_N^{\delta}(Y)(t)\right).
\end{align}
Using Equations \eqref{eqn:operatorform1} and \eqref{eqn:op-taylor}, we get 
\begin{align}
A_N f_p^{\delta} (q)& =N \int_{T_q(T_pM)}\Big[ \di f_p^{\delta}(q)\left( \dfrac{Y-\tilde{\mu}_q}{\sqrt{N}}+\dfrac{1}{N}\tilde{\mu}_q\right)+ \dfrac{1}{2N}\dfrac{\di ^2}{\di t^2}\biggr\vert_{t=t_N(Y)}\left(h_N^{\delta}(Y)(t)\right)\,\Big]
\tilde{\nu}_q (\di Y) \nonumber \\
\label{eqn:second order tight}
& 
=\di f^{\delta}_p(q)(\tilde{\mu}_q)
+ \frac{1}{2}\int_{T_q(T_pM)} \frac{\di ^2}{\di t^2}\biggr\vert_{t=t_N(Y)}\left(h_N^{\delta}(Y)(t)\right) \tilde{\nu}_q (\di Y).
\end{align}
We need to show the equation above is uniformly bounded for all $p\in M$, $q\in T_pM$ and $N\geq 1$.

First we show for each $0<\delta<\delta_0$,
\begin{align}
\label{eqn:first order bound}
\sup_{p\in M}\sup_{q\in T_pM}\lvert\di f^{\delta}_p(q)(\tilde{\mu}_q)\rvert<\infty.
\end{align}
Clearly if $d_a^p(0,q)\leq \frac{\delta}{4}$ or $d_a^p(0,q)\geq \frac{\delta}{2}$, we have
\begin{align*}
\di f^{\delta}_p(q)(\tilde \mu_q)=0.
\end{align*}
For any $q\in T_pM$ such that $\frac{\delta}{4}\leq d_a(0,q)\leq \frac{\delta}{2}$, 
the Finlser metric $F_p|_{T_q(T_pM)}$ and the measure $\tilde{\nu}_q$ are the pull backs of $F$ and $\lbrace \nu_o\rbrace_{o\in M}$ by $\exp_p$, respectively. Thus $F_p(\tilde{\mu}_q)<1$ and $F_p(-\tilde{\mu}_q)\leq C$. It follows that
\begin{align*}
| \tilde{\mu}_q(d^p_a(0,\cdot)) | \leq C\quad \mathrm{if }\ \frac{\delta}{4}\leq d_a^p(0,q)\leq \frac{\delta}{2}.
\end{align*}
The function $d_a^p(0,\cdot)$ is smooth at $q$ For $q\in T_pM$ with $\frac{\delta}{4}\leq d_a(0,q)\leq \frac{\delta}{2}$. Hence Equation 
\eqref{eqn:first order bound} holds by the chain rule.

Now it suffices to prove the integrand in \eqref{eqn:second order tight} is uniformly bounded 
for all $Y\in T_q(T_pM)$, $q\in T_pM$, $p\in M$ and $N\geq 1$. By the construction of $\psi^{\delta}$, for the case
\ba
d_a^p \Big(0,\gamma_{Y-(1-1/\sqrt{N})\tilde \mu_q}(t_N(Y))\Big)\geq \frac{\delta}{2}\quad 
\mathrm{or}\quad d_a^p\Big(0,\gamma_{Y- (1-1/\sqrt{N})\tilde \mu_q}(t_N(Y))\Big)\leq \frac{\delta}{4},
\ea
we have
\begin{align}
\frac{\di ^2}{\di t^2}\biggr\vert_{t=t_N(Y)} h_N^{\delta}(Y)(t) =0.
\end{align}
For the case
\ba
\label{e:dpa}
\frac{\delta}{4}\leq d_a^p\big(0,\gamma_{Y-(1-1/\sqrt{N})\tilde\mu_q}(t_N(Y))\big)\leq \frac{\delta}{2},
\ea
the function  $h_N(Y)(t)$ is smooth on some interval containing $t=t_N(Y)$, because $d_a^p(0,\cdot)$ is smooth on $D_p(\delta)\setminus\lbrace 0 \rbrace$.  
Since $\psi^{\delta}$ is in $\mathcal C^{\infty}_K$, it is sufficient to show the first and second derivatives $h_N(Y)(t)$ with respect to $t$ are uniformly bounded.

To simplify the notations we denote by $\nabla\rho:=\nabla d_a^p(0,\cdot)$ the \emph{Finsler gradient}, see, e.g.\ 
\cite[Equation (3.14) in \S 3.2]{shen2001lectures}.
Following \cite[\S 15.1]{shen2001lectures}, let us define 
\begin{align*} 
& \hat{\mathbf{g}}\coloneqq  \tilde{g}_{\nabla\rho},
\quad r_N(Y)\coloneqq d_a^p\Big(0,\gamma_{Y-(1-1/\sqrt{N})\tilde\mu_q}(t_N)\Big), \\
& \dot{\gamma}_N\coloneqq \frac{\di }{\di t}\Big|_{t=t_N(Y)}\gamma_{Y-(1-1/\sqrt{N})\tilde\mu_q}(t),\quad  Y^{\perp}\coloneqq \dot{\gamma}_N-\hat{\mathbf{g}}(\dot{\gamma}_N,\nabla\rho)\nabla\rho.
\end{align*}
Here $\tilde{g}$ is the fundamental tensor of $F_p$.

Because $\exp_p$ is an isometric immersion on $D_p(\delta)$, on $(T_pM,F_p)$, we also have the uniform elliptic conditions
\begin{align}
\dfrac{1}{C^2}\tilde{g}_v(v,v)\leq \tilde{g}_u(v,v)\leq C^2 \tilde{g}_v(v,v),
\end{align}
for any $0\neq u,v\in T_q(T_pM)$ with $q\in D_p(\delta)$.
It follows that for all $Y\in T_q(T_pM)$ with $F_p(Y)\leq 1$, $q\in D_p(\delta)$ we have
\begin{align*}
&F_p(  \dot{\gamma}_N  )=F_p\Big(Y-(1-1/\sqrt{N})\tilde{\mu}_q\Big)\leq C+1,\\
&F_p(-\dot{\gamma}_N)\leq C+1.
\end{align*}
 This implies  that for all $Y\in T_q(T_pM)$ with $F_p(Y)\leq 1$,
 \begin{align*}
 \Big| \frac{\di}{\di t}\Big|_{t=t_N(Y)} h_N(Y)(t)\Big|\leq C+1,
 \end{align*}
if \eqref{e:dpa} holds.

The second derivative of $h_N(Y)(t)$ can be estimated by the Hessian comparison theorem 
(see Section 15.1 of \cite{shen2001lectures}), 
using the bounded curvature conditions in Definition \ref{definition:bounded geometry}. 
Note that $(D_p(\delta),F_p)$ also has flag curvature and $T$-curvature bounded by $\vert K\vert\leq \lambda$ and $\vert T\vert\leq \lambda$, 
because $\exp_p$ restricted to $(D_p(\delta),F_p)$ is an isometric immersion. 
Since the injective radius of $F_p$ at $0$ is at least $\frac{\delta_c}{4}>\delta$, the Hessian comparison theorem implies:
\begin{align}
\left(\sqrt{\lambda}\cdot\cot(\sqrt{\lambda}\cdot r(Y))-\lambda\right)\hat{\mathbf{g}}(Y^{\perp},Y^{\perp})
\leq \frac{\di^2}{\di t^2}\biggr\vert_{t=t_N(Y)}h_N(Y(t)),\\
\frac{\di ^2}{\di t^2}\Big|_{t=t_N(Y)} h_N(Y(t))\leq \left(\sqrt{\lambda}\cdot\coth(\sqrt{\lambda}\cdot r(Y))+\lambda\right)\hat{\mathbf{g}}(Y^{\perp},Y^{\perp}).
\end{align}
Using the fact $\hat{\mathbf{g}}(\nabla\rho,\nabla\rho)=F^2_p(\nabla\rho)=1$ on $D_p(\delta)\setminus \lbrace 0\rbrace$, we get
\begin{align}
\label{eqn:hessian norm}
\hat{\mathbf{g}}(Y^{\perp},Y^{\perp})&\ =\left\lvert \hat{\mathbf{g}}(\dot{\gamma}_N,\dot{\gamma}_N)-\hat{\mathbf{g}}^2(\dot{\gamma}_N,\nabla\rho)\right\rvert.
\end{align}
If $x\in D_p(\delta)$, for any tangent vectors 
$Y_1,Y_2\in T_x(T_pM)$ with $Y_1\neq 0$, the fundamental inequality 
in Finsler geometry
(see 1.2.16 of \cite{bao2000introduction}) and the inequality $F_p(Y_2)\leq C F_p(-Y_2)$ give
\ba
\vert \tilde{g}_{Y_1}(Y_1,Y_2)\vert\leq C F_p(Y_1)F_p(Y_2).
\ea
Substituting this into \eqref{eqn:hessian norm} and using uniform ellipticity, for $Y$ such that $F_p(Y)\leq 1$ and \eqref{e:dpa} holds, we obtain
\begin{align*}
\hat{\mathbf{g}}(Y^{\perp},Y^{\perp}) & \leq C \tilde{g}_{\dot{\gamma}_N}(\dot{\gamma}_N,\dot{\gamma}_N)+C^2F^2_p(\dot{\gamma}_N)\\
& \leq 2C^2F_p^2\Big(Y-\Big(1-\frac{1}{\sqrt{N}}\Big)\tilde \mu_q\Big)\\
& \leq 2C^2(C+1)^2.
\end{align*}
Since $\frac{\delta}{4}\leq r_N(Y)\leq \frac{\delta}{2}$ and $\delta<\frac{\pi}{2\sqrt{\lambda}}$, we have for all $N\geq 1$:
\begin{align}
 0\leq \cot(\sqrt{\lambda}\cdot r_N(Y))\leq \coth(\sqrt{\lambda}\cdot r_N(Y)). 
\end{align}
Then for all $N\geq 1$ and $Y$ with $F_p(Y)\leq 1$, we have the estimate:
\begin{align}
\left\lvert \frac{\di^2}{\di t^2}\Big|_{t=t_N(Y)} h_N(Y(t))\right\rvert 
& \leq \Big(\sqrt{\lambda}\cdot\coth(\frac{\delta\sqrt{\lambda}}{4}) +\lambda \Big)\hat{\mathbf{g}}(Y^{\perp},Y^{\perp}),\\
  & \leq 2C^2(C+1)^2\Big(\sqrt{\lambda}\cdot\coth(\frac{\delta\sqrt{\lambda}}{4}) +\lambda \Big).
    \end{align}
    
 This shows the second derivative of $h_N(Y)(t)$ evaluated at $t=t_N(Y)$ is also  uniformly and absolutely bounded, 
 if \eqref{e:dpa} holds true. Then there exists some $\tilde C(\delta)>0$ such that condition \ref{condtion:op-bound} holds.
This completes the proof. 
\end{proof}

The family of functions $\{f^\delta_p\}$ will be used now to estimate the first exit time from a $\delta$-balls of the geodesic random walk $\xi^N$.

For each $p\in M$, $N\geq 1$ and $\delta>0$, define the following stopping times for the random walk $\xi^N$ on $M$ and $\xi^{N,p}$ on $T_pM$:
\begin{align}
\tau^{N,\delta}& \coloneqq \inf\lbrace t>0\colon d(\xi^N_t,p)>\delta\rbrace,\\
\tau^{N,\delta}_p& \coloneqq \inf\lbrace   t>0 \colon d^p(\xi^{N,p}_t,0)>\delta  \rbrace,\quad 0\in T_pM.
\end{align}
Now we compare the exit time probabilities of the $\delta$-balls for $\xi^N_t$ and $\xi^{N,p}_t$ for sufficiently large $N$.
 
\begin{lem}
\label{lemma:exit time comparison}
For any $p\in M$, $N\geq 1$ and $\delta$ such that $0<\delta<\delta_0$ and $\frac{2(C+1)}{\sqrt{N}}<\frac{\delta_c}{4}$, we have
\begin{align}
\mathbf{P}_p(\tau^{N,\delta}\leq t)\leq \mathbf{P}_0(\tau^{N,\delta}_p\leq t), \quad \forall t\geq 0.
\end{align}
\end{lem}
\begin{proof}
The geodesic random walks $\xi^{N,p}$ and $\xi^N$ are constructed by randomizing the time of the discrete Markov processes $\zeta^{N,p}$ and $\zeta^N$ using a Poisson process, respectively. 
Hence it suffices to show for each pair $(N,\delta)$ satisfies the condition in the lemma, the following holds
\begin{align}
\label{eqn:walk measure comparison}
\mathbf{P}_0\left(\max_{j\leq k}d^p\left(0,\zeta_j^{N,p} \right) \leq \delta \right)\leq \mathbf{P}_p\left( \max_{j\leq k} d(p,\zeta^N_j )\leq \delta \right)  ,\ \forall p\in M,\ \forall\delta<\delta_0,\ k\geq 0.
\end{align}

For $r>0$, define the closed $\delta$-balls of the symmetrized distances on $T_pM$ and $M$ respectively:
\begin{align*}
B_0^p(r)&= \lbrace y\in T_pM\colon d^p(0,y)\leq  r\rbrace,\\
B_p(r)& =\lbrace q\in M\colon d(p,q)\leq r \rbrace.
\end{align*}
Because $\delta<\delta_0<\frac{\delta_c}{4(C+1)}$, we have $B^p_0(\delta)\subset D_p(\delta_c/2)$. Hence $\exp_p$ maps $(B^p_0(\delta),F_p)$ inside $(B_p(\delta),F)$ by an isometric immersion. Now for each $k\geq 0$ define the following Borel sub-probabilty measures on $B^p_0(\delta)$ and $B_p(\delta)$, respectively.
\begin{align*}
\theta^0_k(\hat{E})& 
=\mathbf{P}_0\left(\zeta^{N,p}_{j}\in B^p_0(\delta),\ 1\leq j\leq k-1,\ \zeta^{N,p}_k\in \hat{E} \right),\ \forall \hat{E}\in \mathcal{B}(B^p_0(\delta));\\
\theta_k(E)& =\mathbf{P}_p\left(\zeta^N_{j}\in B_p(\delta),\ 1\leq j\leq k-1,\ \zeta^N_k\in E \right),\ \forall E\in \mathcal{B}(B_p(\delta)).
\end{align*}
Let $\theta^{p}_k\coloneqq \left(\exp_p\vert_{B^p_0(\delta)}\right)_*\theta^0_k$. 
For integers $N$ such that $\frac{2(C+1)}{\sqrt{N}}<\frac{\delta_c}{4}$, we claim $\theta^p_k\leq \theta_k$ for all $k\geq 0$.

We prove the claim by induction. As for $k=0$, we have $\theta^0_0(\hat{E})=\mathbf{1}_{\hat{E}}(0)$ 
and $\theta_0(E)=\mathbf{1}_E(p)$. Since $\exp_p(0)=p$,
 this claim holds for $k=0$. 
 For simplicity we denote
 \begin{align}
\mathbf{e}_p^N\coloneqq \exp_p\Big(\frac{Y-\mu_p}{\sqrt{N}}+\frac{1}{N}\mu_p\Big).
\end{align}
 For any $q\in B^p_0(\delta)$ and $Y\in T_q(T_pM)$ with $F_p(Y)\leq 1$, the condition 
 $\frac{2(C+1)}{\sqrt{N}}<\frac{\delta_c}{4}$ implies the $N$-scaled geodesic segment $\gamma(t)=\exp_q(t\mathbf{e}^N_q(Y))$ for $t\in [0,1]$ of $F_p$ is 
 mapped by $\exp_p$ to a geodesic segment on $(M,F)$. Also note for $q\in B^p_0(\delta)$, the mean is preserved under $\exp_p$ by
 \begin{align*}
 (\di \exp_p(q))_*(\tilde{\mu}_q)=\mu_{\exp_p(q)}.
 \end{align*}
 Hence for any $q\in B^p_0(\delta)$
 and $Y\in T_q(T_pM)$ with $F_p(Y)\leq 1$ we get
\begin{align}
\label{eqn:measure commutative}
\exp_p \left( \mathbf{e}^N_q(Y)\right)=\mathbf{e}^N_{\exp_p(q)}\left( (\di \exp_p(q))_*(Y) \right).
\end{align}

For each $N\geq 1$, let $P^o(x,\cdot)$ and $P(y,\cdot)$ be the one-step transition probabilities of $\zeta^{N,p}$ and $\zeta^N$, respectively. For $q\in B^p_0(\delta)$, set $q_1=\exp_p(q)$.  Since $\tilde{\nu}_q$ is the pull-back of $\nu_{q_1}$ by $(d\exp_p(q))$, then \eqref{eqn:measure commutative} implies for any Borel set $\hat{E}\subset B^p_0(\delta)$, we have
\begin{align}
P^o(q,\hat{E})=\tilde{\nu}_q\left(\left(\mathbf{e}^N_q\right)^{-1}(\hat{E})\right)\leq \nu_{q_1}\left(\left(\mathbf{e}^N_{q_1}\right)^{-1}(\exp_p(\hat{E}))\right)=P(q_1,\exp_p(\hat{E})).
\end{align}
The inequality in the previous formula appears because the exponential map $\exp_p$ restricted to 
$B^p_0(\delta)$ is not necessarily injective.

In particular, for any Borel $E\subset B_p(\delta)$, we get
\begin{align}
\label{e:PP}
P(q_1,E)\geq P^o(q,(\exp_p)^{-1}(E)\cap B^p_0(\delta)).
\end{align}
 By the Markov property, we have for all $k\geq 1$
\begin{align*}
\theta^o_k(\hat E)&=\int_{B^p_0(\delta)} P^o(y,\hat{E})\theta^o_{k-1}(\di y),\ \forall \hat{E}\in \mathcal{B}(B^p_0(\delta));\\
\theta_k(E)& =\int_{B_p(\delta)} P(x,E)\theta_{k-1}(\di x),\ \forall E\in \mathcal{B}(B_p(\delta)).
\end{align*}
Hence we have the following chain of inequalities:
\begin{align*}
\theta_k(E)& =\int_{B_p(\delta)} P(x,E)\theta_{k-1}(\di x)\\
& \geq \int_{B_p(\delta)} P(x,E)\theta^p_{k-1} (\di x)\\
&= \int_{B^p_0(\delta)} P(\exp_p(y),E)\theta^o_{k-1}(\di y)\\
& \geq \int_{B^p_0(\delta)} P^o(y,(\exp_p)^{-1}(E)\cap B^p_0(\delta))\theta^o_{k-1}(\di y)\\
&=\theta^p_k(E).
\end{align*}
where the first inequality comes for the induction assumption, and the second inequality is due to \eqref{e:PP}.

In particular, for all $k\geq 0$, the inequality
\begin{align*}
\theta^o_k(B^p_0(\delta))= \theta^p_k(B_p(\delta))\leq \theta_k(B_p(\delta))
\end{align*}
implies \eqref{eqn:walk measure comparison} holds. This completes the proof.
\end{proof}
Next, we prove the following estimate on the first exit times of $\xi^N$ from $\delta$-balls on $M$.
\begin{lem}
\label{lemma:operatorlocalness}
For each $\delta>0$ there is $C(\delta)>0$ such that for all $t\geq 0$
\begin{align}
\label{eqn:exist time bound 1}
\sup_{p\in M}\sup_{N\geq 1}\mathbf{P}_p(\tau^{N,\delta}\leq t)\leq C(\delta)t.
\end{align}
In particular,
\begin{align}
\label{eqn:exit time bound 2}
\sup_{p\in M}\sup_{N\geq 1}\mathbf{E}_p\ex^{-\tau^{N,\delta}} <1.
\end{align} 
\end{lem} 
\begin{proof}
We adapt ideas from \cite{Kunita-95} by Kunita.
The proof is divided into two steps. First we consider the exit times for the lifted random walks $\xi^{N,p}$, 
and we show
\begin{align}
\label{e:ssP}
\sup_{p\in M}\sup_{N\geq 1}\mathbf{P}_0(\tau^{N,\delta}_p\leq t)\leq \tilde{C}(\delta)t,
\end{align}
where $0\in T_pM$ and $\tilde{C}(\delta)>0$ is the constant defined in Lemma \ref{lemma:generator bounds}.
Next we use Lemma \ref{lemma:exit time comparison} and \eqref{e:ssP} to prove \eqref{eqn:exist time bound 1}, and \eqref{eqn:exit time bound 2} directly follows from \eqref{eqn:exist time bound 1}.

It suffices to show \eqref{e:ssP} for sufficiently small $\delta$. For any $\delta\in (0,\delta_0)$ and $p\in M$, let $\lbrace f^{\delta}_p\rbrace$ be the family of functions constructed in Lemma \ref{lemma:generator bounds}. Using $0\leq f^{\delta}_p\leq 1$ and $f^{\delta}_p(0)=1$, we have
\begin{align*}
\mathbf{P}_0(\tau^{N,\delta}_p\leq t)& =1-\mathbf{P}_0(\tau^{N,\delta}_p>t)\\ 
& \leq 1-\mathbf{E}_0\Big[ \mathbb{I}(\tau^{N,\delta}_p>t)f^{\delta}_p(\xi^{N,p}(\tau^{N,\delta}_p\wedge t))\Big], \\
& =1- \mathbf{E}_0\Big[\Big(1- \mathbb{I}(\tau^{N,\delta}_p\leq t)\Big)f^{\delta}_p\Big(\xi^{N,p}(\tau^{N,\delta}_p\wedge t)\Big)\Big]\\
&= f^{\delta}_p(0)-\mathbf{E}_0 f^{\delta}_p\Big(\xi^{N,p}(\tau^{N,\delta}_p\wedge t)\Big)+ \mathbf{E}_0\Big[ \mathbb{I}(\tau^{N,\delta}_p\leq t)f^{\delta}_p\Big(\xi^{N,p}(\tau^{N,\delta}_p\wedge t)\Big)\Big]\\
&= f^{\delta}_p(0)-\mathbf{E}_0 f^{\delta}_p\Big(\xi^{N,p}(\tau^{N,\delta}_p\wedge t)\Big)+ \mathbf{E}_0\Big[ \mathbb{I}(\tau^{N,\delta}_p\leq t)
f^{\delta}_p\Big(\xi^{N,p}(\tau^{N,\delta}_p)\Big)\Big].
\end{align*}
Here the notation  $a\wedge b:=\min\{a,b\}$ is standard in the theory of stochastic processes.  
Taking into account that $f^{\delta}_p\left(\xi^{N,p}\left(\tau^{N,\delta}_p\right)\right)=0$ and 
applying the Dynkin formula to above, we obtain
\begin{align*}
\mathbf{P}_0(\tau^{N,\delta}_p\leq t) \leq -\mathbf{E}_0\int^{\tau^{N,\delta}_p\wedge t}_0 A_{N,p} f^{\delta}_p(\xi^{N,p}_s)\,\di s
 \leq \sup_{N\geq 1}\lVert A_{N,p} f^{\delta}_p\rVert\cdot t= \tilde{C}(\delta)t,\ \forall p\in M.
\end{align*}
Let $N_0$ be the smallest positive integer such that $\dfrac{2(C+1)}{\sqrt{N_0}}< \dfrac{\delta_c}{4}$.  
For $N\leq N_0$, we have always have
\begin{align}
\mathbf{P}_p(\tau^{N,\delta}\leq t)\leq \mathbf{P}(Q(Nt)>0)\leq Nt\leq N_0t.
\end{align}  
This together with Lemma \ref{lemma:exit time comparison} proves the inequality \eqref{eqn:exist time bound 1} by setting
$C(\delta)=\max\{\tilde C(\delta),N_0\}$.

Furthermore, for $t_*=\dfrac{1}{2C(\delta)}>0$
\begin{align*}
\mathbf{E}_p\ex^{-\tau^{N,\delta}}& 
=\mathbf{E}_p\mathbb{I}(\tau^{N,\delta}\leq t_*)\ex^{-\tau^{N,\delta}}+\mathbf{E}_p\mathbb{I}(\tau^{N,\delta}> t_*)\ex^{-\tau^{N,\delta}}\\
& \leq \mathbf{P}_p(\tau^{N,\delta}\leq t_*)+\ex^{-t_*}\left( 1-\mathbf{P}_p(\tau^{N,\delta}\leq t_*) \right)\\
& = \ex^{-t_*}+(1-\ex^{-t_*})\mathbf{P}_p(\tau^{N,\delta}\leq t_*).
\end{align*}
Note that due to \eqref{eqn:exist time bound 1}, we have $\mathbf{P}_p(\tau^{N,\delta}\leq t_*)\leq C(\delta)t_*\leq \dfrac{1}{2}$. Hence we obtain
\begin{align*}
\mathbf{E}_p\ex^{-\tau^{N,\delta}}\leq \ex^{-t_*}+\dfrac{1-\ex^{-t_*}}{2}=\dfrac{1+\ex^{-t_*}}{2}<1.
\end{align*}
This proves the second inequality of the lemma.
\end{proof}
Now we are ready to show the Aldous criteria hold in our situation.

\begin{lem}[Aldous criteria]
\label{l:AC}
For any initial point $p\in M$, any $T>0$, $\delta>0$, and any $(\mathcal{F}^N)$-stopping times $0\leq \tau\leq T$, we have
\begin{align}
\label{eqn:Aldous}
\lim_{s\rightarrow 0}\limsup_{N\rightarrow \infty}\sup_{\tau}\sup_{h\in[0,s]}\mathbf{P}_p\left( d(\xi^N_{\tau},\xi^N_{\tau+h})>\delta \right)=0
\end{align}
\end{lem}

\begin{proof}
Let $\delta,s>0$ be fixed. For each $N\geq 1$, $p\in M$, $h\in[0,s]$ and a stopping time $\tau$ we have
\begin{align*}
\mathbf{P}_p\left(d(\xi^N_{\tau},\xi^N_{\tau+h})>\delta\right)& =\mathbf{E}_p\biggr [ \mathbb{I}\left(d(\xi^N_{\tau},\xi^N_{\tau+h})>\delta)\right)\biggr ],\\
& =\mathbf{E}_p\biggr [\mathbf{E}\biggr [\mathbb{I}\left(d(\xi^N_{\tau},\xi^N_{\tau+h})>\delta)\right)\vert \mathcal{F}^N_{\tau}\biggr ]\biggr ],\\
& =\mathbf{E}_p\biggr [ \mathbf{P}\left( d(\xi^N_{\tau},\xi^N_{\tau+h}) >\delta\vert \mathcal{F}^N_{\tau}\right)\biggr ].
\end{align*}
The strong Markov property of $\xi^N$ yields
\begin{align*}
\mathbf{E}_p\biggr [ \mathbf{P}\left( d(\xi^N_{\tau},\xi^N_{\tau+h}) >\delta\vert \mathcal{F}^N_{\tau}\right)\biggr ]
=\mathbf{E}_p\mathbf{P}_{\xi^N_{\tau}}\biggr (d(\xi^N_0,\xi^N_h)>\delta \biggr )
\leq \mathbf{E}_p\Big[\sup_{q\in M} \mathbf P_q(\tau^{N,\delta}\leq h)\Big].
\end{align*}
Thus by Lemma \ref{lemma:operatorlocalness}, we have
\begin{align}
\mathbf{P}_p\left(d(\xi^N_{\tau},\xi^N_{\tau+h})>\delta\right)\leq C(\delta)h\leq C(\delta)s.
\end{align}
Taking supremums  and letting $s\rightarrow 0$, we obtain the limit in Equation \eqref{eqn:Aldous}.
\end{proof}

Next we show the family of processes $\lbrace \xi^N\rbrace$ has compact containment property as follows.

\begin{lem}[compact containment condition]
\label{l:cc}
For any $\varepsilon>0$, $T\geq 0$ and $p\in M$ there is a compact neighborhood $K_{\varepsilon}(p)\subseteq M$ of $p$ such that
\begin{align}
\inf_N\mathbf{P}_p\Big(\xi^N_t\in K_{\varepsilon}(p),\ t\in [0,T]\Big)\geq 1-\varepsilon.
\end{align}
\end{lem}
\begin{proof}
Let us define the following sequence of exit times
\begin{align*}
\tau^N_0\ & \coloneqq 0,\\
\tau^N_k & \coloneqq \inf\left\lbrace s>\tau^N_{k-1}\colon d\left(\xi^N_s,\xi^N_{\tau^N_{k-1}}\right)> 1\right\rbrace,\ k\geq 1,
\end{align*}
(as usual, we set $\inf\emptyset=+\infty$).

By Lemma \ref{lemma:operatorlocalness}, there exists some constant $c_1\in(0,1)$ such that
\begin{align*}
\sup_{p\in M}\sup_N\mathbf{E}_p \ex^{-\tau^N_1}=\sup_{p\in M}\sup_N\mathbf{E}_p \ex^{-\tau^{N,\delta_1}}\leq c_1<1
\end{align*}
Then for $k\geq 1$ and $\forall N\geq 1$, the strong Markov property yields
\begin{align*}
\mathbf{E}_p\ex^{-\tau^N_k}& = \mathbf{E}_p \big [\ex^{-\tau^N_{k-1}}\cdot \ex^{\tau^N_{k-1}-\tau^N_k}\big ],\\
& = \mathbf{E}_p \biggr [\ex^{-\tau^N_{k-1}}\cdot \mathbf{E}\big [\ex^{\tau^N_{k-1}-\tau^N_k}\vert \mathcal{F}^N_{\tau^N_{k-1}}\big ]\biggr ],\\
& =  \mathbf{E}_p \biggr [\ex^{-\tau^N_{k-1}} \cdot \mathbf{E}_{\xi^N_{\tau^N_{k-1}}} \ex^{-\tau^N_1}\biggr ],\\
& \leq c_1\cdot \mathbf{E}_p\ex^{-\tau^N_{k-1}}\leq c_1^k.
\end{align*}
For any $\varepsilon>0$ and $T\geq 0$, define
\begin{align}
\label{eqn:compact containment integer}
k_{\varepsilon}\coloneqq \left\lceil\dfrac{\ln{\varepsilon}-T}{\ln{c_1}}\right\rceil.
\end{align}
Then the exponential Markov inequality gives $\forall p\in M$, $\forall N\geq 1$:
\begin{align}
\mathbf{P}_p\left( \tau^N_{k_{\varepsilon}}\leq T\right)=\mathbf{P}_p \left(\ex^{-\tau^N_{k_{\varepsilon}}}\geq \ex^{-T} \right)
\leq \ex^{T}\mathbf{E}_p\ex^{-\tau^N_{k_{\varepsilon}}}\leq \ex^{T} c_1^{k_{\varepsilon}}\leq \varepsilon.
\end{align}
By construction and the triangle inequality we have that for each $k\geq 1$, each $N\geq 1$
\begin{align}
 &d\left(\xi^N_{\tau^N_{k}},\xi^N_{\tau^N_{k-1}}\right)\leq 1+ \sup_N\sup_p d(p,\zeta^N_1).
\end{align}
We estimate the last term:
\begin{align}
 d(p,\zeta^N_1)&\leq \sup_{Y\in D_pM}d\Big(p,\exp_p\Big(\frac{Y-\mu_p}{\sqrt N}+\frac{\mu_p}{N}\Big)\Big)\\
 &\leq \sup_{Y\in D_pM}\dfrac{C}{\sqrt{N}}{}\cdot F\Big(Y-\Big(1-\frac{1}{\sqrt{N}}\Big)\mu_{p}\Big)\\
 & \leq C(C+1).
\end{align}
Thus we have for $k\geq 1$ 
\begin{align}
d\left(\xi^N_0,\xi^N_{\tau^N_k}\right)\leq k\Big(1+C(C+1)\Big). 
\end{align}

Now for $p\in M$, $\varepsilon>0$, and $k_{\varepsilon}$ defined as in Equation 
\eqref{eqn:compact containment integer}, consider the closed ball
\begin{align}
\label{e:K}
K_p(\varepsilon)\coloneqq \lbrace q\in M \colon d(p,q)\leq  R(\varepsilon,T)\rbrace,
\end{align}
with radius 
\ba
\label{e:R}
R(\varepsilon,T)=k_{\varepsilon}(1+C(C+1))+1.
\ea
Then $K_p(\varepsilon)$ is closed and forward bounded, hence it is compact by Hopf--Rinow theorem.

Eventually we get that $\forall p\in M$ and $N\geq 1$:
\begin{equation}
\begin{aligned}
\mathbf{P}_p&\left( \xi^N_t\notin K_p(\varepsilon)\ \mathrm{for\ some\ } t\leq T \right)\\
& \leq \mathbf{P}_p\left( \tau^N_{k_{\varepsilon}}\leq T \right)+ \mathbf{P}_p\left( \xi^N_t\notin K_p(\varepsilon)\ \mathrm{for\ some\ } t\leq T,\ \tau^N_{k_{\varepsilon}}> T \right)\\
& \leq \varepsilon,
\end{aligned}
\end{equation}
since the last summand equals to zero by construction of the set $K_p(\varepsilon)$. 
This finishes the proof of compact containment condition.
\end{proof}
So far, we have proved the sequence $\lbrace \xi^N\rbrace$ satisfies both Aldous criteria and the compact containment condition. Thus this sequence is tight. It is well known tightness implies being relatively compact. Thus any subsequence of $\lbrace \xi^N\rbrace$ has a further subsequence converging weakly to some process $\xi$ on $M$.

We close this section by showing any limit process of $\lbrace \xi^N\rbrace$ has continuous paths almost surely.

\begin{prp}
\label{prop:continuous path}
Any limit point $\xi$ of geodesic random walks $\lbrace \xi^N\rbrace$ is a.s.\ continuous.
\end{prp}
\begin{proof}
The uniform elliptic condition implies that the jump sizes of the geodesic random walks $\xi^N$ converge to zero uniformly as $N\to\infty$,  since
\ba
d(\xi^N_{t-},\xi^N_t)\leq \frac{C+1}{\sqrt{N}},\ \forall t\in[0,\infty).
\ea
Hence the statement follows immediately from Theorem 3.10.2 of \cite{EthierK-86}.
\end{proof}

\subsection{Convergence of geodesic random walks.}
In this section, we give the proof of Theorem~\ref{thm:convergence main theorem}. We already know the sequence $\lbrace \xi^N\rbrace$ is relatively compact. To show the weak convergence, it remains to prove all limit points of $\lbrace \xi^N\rbrace$ have the same law. This is achieved by showing any limit point of this sequence is a solution to a well-posed martingale problem.

We first need the following lemma. Recall that $A$ defined in  \eqref{eqn:diffusion generator} is the limit of the generators $A_N$ .
\begin{lem}
\label{lemma:martingale}
For any $p\in M$, any limit point $\xi$ of $\lbrace \xi^N\rbrace$ and any $f\in \mathcal{C}^{\infty}_K$, we have 
\ba
\label{e:martP}
f(\xi_t)-f(p)-\int_0^t Af(\xi_s)\,\di s,\quad \forall t\geq 0,
\ea
is a martingale.
\end{lem}
\begin{proof}
It suffices to show that for any $l\geq 1$, any $h_1,\dots,h_l\in \mathcal{C}_b(M)$, any $0\leq s\leq t$,
$s_1,\dots,s_l\in [s,t]$,
and any $f\in \mathcal{C}^{\infty}_K$, the following holds.
\begin{align}
\label{eq:martA}
\mathbf{E}\Big[ \Big(  f(\xi_t)-f(\xi_s)-\int_s^t Af(\xi_r)\,\di r \Big) \prod_{j=1}^l h_j(\xi_{s_j})   \Big]=0.
\end{align} 
Since $\xi^N$ is a Markov process for all $N\geq 1$, it follows for all $0\leq s\leq t$ that 
\ban
f(\xi^N_t)-f(\xi^N_s)-\int^t_s A_N f(\xi^N_r)\,\di r
\ean
is a martingale. Hence for each $N\geq 1$
\begin{align}
\mathbf{E}\Big[ \Big(  f(\xi^N_t)-f(\xi^N_s)-\int_s^t A_Nf(\xi^N_r)\,\di r \Big) \prod_{j=1}^l h_j(\xi^N_{s_j})   \Big]=0.
\end{align}
Separate the formula \eqref{eq:martA} into two terms, and let $\lbrace\xi^{N_k}\rbrace$ be a subsequence converging weakly to $\xi$.   Since $\xi$ has continuous paths almost surely, the  finite dimensional
distributions of $\xi^{N_k}$ always converge weakly to those of $\xi$ (Theorem 3.7.8 of \cite{EthierK-86}). Thus we have
\ba
\mathbf{E}\Big[ \Big(  f(\xi_t)-f(\xi_s)\Big) \prod_{j=1}^l h_j(\xi_{s_j})   \Big]
=\lim_{N_k\to\infty}\mathbf{E}\Big[ \Big(  f(\xi^{N_k}_t)-f(\xi^{N_k}_s)\Big) \prod_{j=1}^l h_j(\xi^{N_k}_{s_j})   \Big].
\ea
Furthermore,
\ba
\mathbf{E}\Big[\int_s^t A_{N_k}f(\xi^{N_k}_r)\,\di r \cdot \prod_{j=1}^l h_j\left(\xi^{N_k}_{s_j}\right)   \Big]
&=\mathbf{E}\Big[\int_s^t A f(\xi^{N_k}_r)\,\di r \cdot \prod_{j=1}^l h_j\left(\xi^{N_k}_{s_j}\right)   \Big]\\
&+\mathbf{E}\Big[\int_s^t (A_{N_k}-A)f(\xi^{N_k}_r)\,\di r \cdot \prod_{j=1}^l h_j\left(\xi^{N_k}_{s_j}\right)   \Big]
\ea
and the latter summand vanishes as $N_k\to\infty$ because the functions $h_j$ are bounded
and by Proposition \ref{prop:limit generator}  
\begin{align*}
\lim_{N\rightarrow\infty}\lVert (A_{N_k}-A)f\rVert =0,\quad \forall f\in \mathcal{C}^{\infty}_K.
\end{align*}
To treat the first term, since $x\mapsto Af(x)$ is continuous and bounded, we have for each $r\in[s,t]$
\ba
\lim_{N_k\rightarrow\infty} \mathbf{E}\Big[ A f(\xi^{N_k}_r) \cdot \prod_{j=1}^l h_j\left(\xi^{N_k}_{s_j}\right)   \Big]= \mathbf{E}\Big[ A f(\xi_r) \cdot \prod_{j=1}^l h_j(\xi_{s_j})   \Big].
\ea
Thus by Fubini's and Lebesgue's theorems  we get 
\ba
\lim_{N\to\infty} \mathbf{E}\Big[\int_s^t A f(\xi^N_r)\,\di r \cdot \prod_{j=1}^l h_j(\xi^N_{s_j})   \Big]
&=\int_s^t \lim_{N\to\infty} \mathbf{E}\Big[ A f(\xi^N_r) \cdot \prod_{j=1}^l h_j(\xi^N_{s_j})   \Big]\,\di r\\
&= \mathbf{E}\Big[\int_s^t A f(\xi_r)\,\di r \cdot \prod_{j=1}^l h_j(\xi_{s_j})   \Big].
\ea
and \eqref{eq:martA} is established.
\end{proof}

\begin{prp}
The martingale problem \eqref{e:martP} has a unique solution which is stochastically complete. 
Hence the sequence $\lbrace \xi^N\rbrace$ converges weakly.
\end{prp}
\begin{proof}
The well-posedness follows from the well-posedness of the martingale problem in $\mathbb R^m$. Indeed, in any chart $U$, the generator $A$ is a second-order 
strongly elliptic operator with smooth coefficients. We can extend the generator on $U^c$ such that its
coefficients are uniformly Lipschitz.
Then the martingale problem is well-posed, e.g.\ by Theorem 5.1.4 in \cite{Stroock}. 
By Theorem 4.6.1 of \cite{EthierK-86}, the stopped martingale is also well posed for any initial distribution. 
The localized solutions in countably many charts can glued together by Lemma 4.6.5 and Theorem 4.6.6 in \cite{EthierK-86}, see also Section 4.11 in
\cite{Kolokoltsov-11}.
\end{proof}

In summary, we have shown the sequence $\lbrace \xi^N\rbrace$ converges weakly to some process $\xi$ on $M$ 
which is a solution to a well-posed martingale problem. This completes the proof of Theorem \ref{thm:convergence main theorem}.

Eventually let us also prove, since it is an important and useful property,  that the  limit process is  Feller.
\begin{prp}
The limit process $\xi$ is Feller, i.e.\ its semigroup preserves $\mathcal{C}_0(M)$. 
\end{prp}
\begin{proof}
In any chart, $\xi$ is a non-degenerate diffusion with smooth coefficients, hence its semigroup maps $\mathcal C_0(M)$ to $\mathcal {C}(M)$.

Denote $(T_t)_{t\geq 0}$ the semigroup of $\xi$ as usual. Let $f\in \mathcal C_0(M)$, $t>0$ and $\varepsilon>0$ be fixed. Choose a compact set $C_\varepsilon$ such that
$|f(x)|\leq \varepsilon$ for $x\notin C_\varepsilon$. Define $R(\varepsilon,t)$ as in \eqref{e:R} in Lemma~\ref{l:cc}.
By this lemma,
for any $p\in M$ such that $d(p,C_\varepsilon)>R(\varepsilon,t)$, we have 

\ba
|\mathbf E_p f(\xi_t)|&\leq \mathbf E_p |f(\xi_t)|\mathbb I(\tau^R\leq t)+\mathbf E_p |f(\xi_t)|\mathbb I(\tau^R> t)\\
&\leq \|f\|\cdot \mathbf P_p(\tau^R\leq t) + \varepsilon\leq  (\|f\| +1) \varepsilon.
\ea
Thus $(T_t)(f)$ vanishes at infinity for $f\in \mathcal{C}_0$.
As $\xi$ is a limit point of $\lbrace \xi^N\rbrace$, Lemma~\ref{lemma:operatorlocalness} implies
\begin{align*}
\sup_{p\in M} \mathbf{P}_p(d(p,\xi_t)>\delta)\leq C(\delta)t,\ \forall \delta>0,\ \forall t\geq 0.
\end{align*}
The strong continuity of the semigroup $(T_t)$ follows.
 \end{proof}
 
%
%

\vspace{2cm}

\begin{minipage}{.33\textwidth} {\small
\begin{tabbing}
xxxxxxxxxxxxxxxxxxxxxxxxxxxxxxxxx\=
\kill
Tianyu Ma   \>    \\
 Institut für Mathematik \>  \\
Friedrich--Schiller Universit\"at Jena \>  \\
07737 Jena, Germany  \>  \\
 
{\tt tianyuzero.ma@alum.utoronto.ca} \>
 
\end{tabbing}
}\end{minipage}\begin{minipage}{.33\textwidth}
{\small
\begin{tabbing}
xxxxxxxxxxxxxxxxxxxxxxxxxxxxxx\=
\kill
Vladimir S. Matveev  \>    \\
 Institut für Mathematik \>  \\
Friedrich--Schiller Universit\"at Jena \>  \\
07737 Jena, Germany  \>  \\
{\tt vladimir.matveev@uni-jena.de} \>
 
\end{tabbing}
}\end{minipage}\begin{minipage}{.33\textwidth}{\small
\begin{tabbing}
xxxxxxxxxxxxxxxxxxxxxxxxxxxxxxxx\=
\kill
Ilya Pavlyukevich \>    \\
 Institut für Mathematik \>  \\
Friedrich--Schiller Universit\"at Jena \>  \\
07737 Jena, Germany  \>  \\
{\tt ilya.pavlyukevich@uni-jena.de } \>
 
\end{tabbing}
}\end{minipage}
\end{document}